\documentclass[11pt]{amsart} \textwidth=14.5cm \oddsidemargin=1cm
\usepackage{amsmath, amstext, amsbsy, amssymb, amscd, mathrsfs}
\usepackage{amsmath}
\usepackage{amsxtra}
\usepackage{amscd}
\usepackage{amsthm}
\usepackage{amsfonts}
\usepackage{amssymb}
\usepackage{eucal}
\usepackage{color}
\usepackage[all]{xy}
\usepackage{bbding}
\usepackage{pmgraph}
\usepackage{stmaryrd}
\usepackage{graphicx}
\usepackage{bm}

\newtheorem{theorem}{Theorem}[section]

\newtheorem{lemma}[theorem]{Lemma}
\newtheorem{proposition}[theorem]{Proposition}
\newtheorem{corollary}[theorem]{Corollary}
\theoremstyle{definition}
\newtheorem{definition}[theorem]{Definition}
\newtheorem{example}[theorem]{Example}
\newtheorem{xca}[theorem]{Exercise}
\theoremstyle{remark}
\newtheorem{remark}[theorem]{Remark}

\definecolor{A}{rgb}{.75,1,.75}

\numberwithin{equation}{section}

\newcommand{\aaa}{\mf a^+_\infty}
\newcommand{\bbb}{\mf b_\infty}
\newcommand{\ccc}{\mf c_\infty}
\newcommand{\ddd}{\mf d_\infty}

\newcommand{\ch}{\text{ch} }
\newcommand{\C}{ \mathbb C}
\newcommand{\etabf}{\underline{\eta}}
\newcommand{\Cx}{\C[\x,\etabf]}

\newcommand{\ep}{ \epsilon}
\newcommand{\End}{\text{End}}
\newcommand{\h}{\mathfrak{h}}
\newcommand{\ind}{\text{Ind}}
\newcommand{\ga}{\mathfrak{g}}
\newcommand{\gl}{\mathfrak{gl}}
\newcommand{\glmn}{\mathfrak{gl}(m|n)}
\newcommand{\gr}{\text{gr} }

\newcommand{\Hom}{\text{Hom} }
\newcommand{\hf}{\frac12}
\newcommand{\la}{\lambda}
\newcommand{\La}{\Lambda}
\newcommand{\n}{\mathfrak n}
\newcommand{\N}{\mathbb N}
\newcommand{\ov}{\overline}
\newcommand{\osp}{\mf{osp}}
\newcommand{\spo}{\mf{spo}}
\newcommand{\ev}[1]{{#1}_{\bar{0}}}
\newcommand{\od}[1]{{#1}_{\bar{1}}}
\newcommand{\Pmn}{\mathcal P(m|n)}
\newcommand{\Pdmn}{\mathcal P(d, m|n)}
\newcommand{\Pdn}{\mathcal P(d, n)}
\newcommand{\Pdm}{\mathcal P(d, m)}

\newcommand{\R}{ \mathbb R}
\newcommand{\Sp}{\text{Sp}}
\newcommand{\sgn}{\text{sgn}}
\newcommand{\sgl}{\mathfrak{sl}}
\newcommand{\str}{\text{str}}
\newcommand{\tr}{\text{tr}}
\newcommand{\Tr}{\mathfrak{Tr}}
\newcommand{\wt}{\widetilde}
\newcommand{\x}{ \underline{x}}
\newcommand{\y}{ \underline{y}}
\newcommand{\Z}{ \mathbb Z }

\newcommand{\mc}{\mathcal}
\newcommand{\mf}{\mathfrak}

{\vskip-\lastskip\medskip
  \noindent
  {\em #1.}\enspace
  }%
{\qed\par\medskip
  }

\title[Dualities for Lie superalgebras]
{Dualities for Lie superalgebras}

\author{Shun-Jen Cheng}
\address{Institute of Mathematics, Academia Sinica, Taipei,
Taiwan 10617} \email{chengsj@math.sinica.edu.tw}

\author{Weiqiang Wang}
\address{Department of Mathematics, University of Virginia, Charlottesville, VA 22904}
\email{ww9c@virginia.edu}

\begin{document}

\maketitle

\begin{abstract}
We explain how Lie superalgebras of types $\gl$ and $\osp$ provide a natural framework
generalizing the classical Schur and Howe dualities. This exposition includes a
discussion of super duality, which connects the parabolic categories $O$ between
classical Lie superalgebras and Lie algebras. Super duality provides a conceptual
solution to the irreducible character problem for these Lie superalgebras in terms of the
classical Kazhdan-Lusztig polynomials.
\end{abstract}

\setcounter{tocdepth}{1} \tableofcontents

\section*{Introduction}

\subsection{} \label{sec:obstacle}
The study of Lie superalgebras, supergroups, and their representations was largely
motivated by supersymmetry in mathematical physics, which puts bosons and fermions on the
same footing. An earlier achievement is the Cartan-Killing type classification of
finite-dimensional simple complex Lie superalgebras by Kac \cite{K1} (also cf. \cite{SNR}
for an independent classification of the so-called classical Lie superalgebras). The
most important basic classical Lie superalgebras consist of
infinite series of types $\mf{sl}, \osp$. The basic classical
Lie superalgebras afford root systems, Dynkin diagrams, Cartan subalgebras, triangular
decomposition, Verma modules, category $O$, and so on. There has been much work on
representation theory of Lie superalgebras (in particular, basic classical) in the last
three decades, but conceptual approaches have been lacking until recently.

\subsection{}

The aim of these lecture notes is to explain three different kinds of dualities for Lie
superalgebras:
$$
\text{Schur duality, } \quad \text{Howe duality,} \quad
\text{and Super duality.}
$$
In the superalgebra setting, the first (i.e.~Schur) duality was formulated by Sergeev,
and the latter two dualities have been largely developed by the authors and their
collaborators. These lecture notes are also intended to serve as a road map for a
forthcoming book by the authors.

The Schur-Sergeev duality is an interplay between Lie superalgebras and the symmetric
groups which incorporates the trivial and sign modules in a unified framework. On the
algebraic combinatorial level, there is a natural super generalization of the notion of
semistandard tableaux which is a hybrid of the traditional version and its conjugate
counterpart.

It has been observed that much of the study of the classical invariant theory on
polynomial algebras has parallels for exterior algebras, and both admit reformulation and
extension in the theory of Howe's reductive dual pairs. Lie superalgebras allow a uniform
treatment of Howe duality on the polynomial and exterior algebras.

Super duality has a different flavor. It views representation theories of Lie
superalgebras and Lie algebras as two sides of the same coin, and it is an unexpected
yet powerful approach developed in the past few years which allow us to overcome various
superalgebra difficulties. We give an exposition on the new development on super duality
which is an equivalence between parabolic categories $\mc O$ of Lie superalgebras and Lie
algebras. Super duality provides a conceptual solution to the long-standing irreducible
character problem for a wide class of modules over (a wide class of) Lie superalgebras in
terms of Kazhdan-Lusztig polynomials. This is achieved despite the fact that there are no
obvious Weyl groups controlling the linkage for super representation theory.

\subsection{}

In Section~\ref{sec:abc}, we give some basic constructions and
structures of the general linear and the ortho-symplectic Lie
superalgebras. We emphasize the super phenomena that are not
observed in the ordinary Lie algebra setting, such as odd roots,
non-conjugate Borel subalgebras, and so on. In
Section~\ref{sec:typical}, we present Kac's classification of
finite-dimensional simple $\ga$-modules \cite{K2}. The
classification is very easy for type $A$, but nontrivial for
$\osp$. In the latter case we explain a new odd reflection
approach by Shu and the second author \cite{SW}, using a  more
natural labeling of these modules by hook partitions. We note that
odd reflection is also one of the main technical tools in super
duality. In addition, we  present the {\em typical}
finite-dimensional irreducible character formula, following
\cite{K2}.

\subsection{}
The classical Schur duality relates the representation theory of the general linear Lie
algebras and that of the symmetric groups. In Section~\ref{sec:Schur}, we explain
Sergeev's generalization \cite{Sv} of Schur duality for the general linear Lie
superalgebras $\glmn$ (also see Berele and Regev \cite{BeR} for additional insight and
detail). More precisely, we establish a double centralizer theorem for the actions of
$\glmn$ and the symmetric group $\mf S_d$ in $d$ letters on the tensor space
$(\C^{m|n})^{\otimes d}$. We then provide an explicit multiplicity-free decomposition of
the tensor space into a $U(\glmn) \otimes \C \mf S_d$-modules. We further present a
simple formula obtained in our latest work with Lam \cite{CLW} for extremal weights in a
simple polynomial $\glmn$-module  with respect to all Borel subalgebras, which has an
explicit diagramatic interpretation from a Young diagram.

\subsection{}
Howe's theory of reductive dual pairs \cite{H1, H2} can be viewed
as a representation theoretic reformulation and extension of the
classical invariant theory (see Weyl \cite{We}). For example, the
first fundamental theorem on invariants for classical groups are
reformulated in terms of double centralizer properties of two
classical Lie groups/algebras. One advantage of Howe duality is
that it allows natural generalizations to classical Lie
groups/algebras (and superalgebras) other than type $A$.

We mainly use two examples of dual pairs to illustrate the main ideas of Howe duality and
the new phenomena of superalgebra generalizations. For more detailed case study of Howe
duality for Lie superalgebras, we refer to the original papers \cite{BPT, CW1, CW2, CW3,
CL1, CLZ, CZ, CKW, LZ, Sv2}. In Section~\ref{sec:Howe}, we formulate the $(\glmn,
\gl(d))$-Howe duality and find the highest weight vectors for each isotypical component
in the corresponding multiplicity-free decomposition. In Section~\ref{sec:Howe app} we
present the $(\Sp(d), \osp(2m|2n))$-Howe duality and its multiplicity-free decomposition.
The application of the Howe duality to irreducible characters over Lie superalgebras
follows the simpler approach in our work with Kwon \cite{CKW} (which uses Howe duality
for infinite-dimensional Lie algebras \cite{Wa}).

\subsection{}

We recall some truly super phenomena that have been the main obstacles towards a
better understanding of super representation theory:

\begin{enumerate}
\item There exist odd roots as well as non-conjugate Borel subalgebras for a Lie
    superalgebra. A homomorphism between Verma modules may not be injective.

\item The linkage in category $\mc O$ of modules for a Lie superalgebra is NOT
    controlled by the Weyl group of $\ev \ga$; see e.g.  $\gl(1|1)$.

\item There is no uniform Weyl-type irreducible finite-dimensional character formula
    for Lie superalgebras.

\item The super geometry behind super representation theory is still inadequately
    developed.
\end{enumerate}

In light of these super phenomena, it was a rather unexpected discovery \cite{CWZ, CW4},
which was partly inspired by Brundan \cite{Br},
that there exists a (conjectural) equivalence of categories between Lie algebras and Lie
superalgebras of type $A$ (at a certain suitable limit at infinity), which was termed
{\em Super Duality}. This conjecture in the full generality of \cite{CW4} has been proved in
\cite{CL}, which in particular offers an elementary and conceptual solution to the
character problem for all finite-dimensional simple modules and for a large class of
infinite-dimensional simple highest weight modules over Lie superalgebras of type $A$.

Super duality has been subsequently formulated and established between various Lie
superalgebras of type $\osp$ and the corresponding classical Lie algebras in our very
recent work with Lam \cite{CLW}. This in particular offers a conceptual solution of the
irreducible character problem for a wide class of modules, which include all
finite-dimensional irreducibles, of Lie superalgebras of type $\osp$ in terms of
Kazhdan-Lusztig polynomials for classical Lie algebras \cite{KL, BB, BK} (for more on
Kazhdan-Lusztig theory see Tanisaki's lectures \cite{Ta}). In addition, it follows easily
from the approach of \cite{CL, CLW} that the $\mf u$-homology groups (or Kazhdan-Lusztig
polynomials in the sense of Vogan \cite{Vo}) match perfectly between classical Lie
superalgebras and the corresponding classical Lie algebras.  This generalizes earlier
partial results in this direction from Schur or Howe duality approach \cite{CZ1, CK,
CKW}. The super duality as outlined above is explained in Section~\ref{sec:SD}.

Let us put the super duality work explained above in perspective.
{\em Finite-dimensional} irreducible characters for $\glmn$ have
been also obtained earlier in two totally different approaches by
\cite{Sva} and \cite{Br}. The mixed algebraic and geometric
approach of  Serganova has been extended very recently in
\cite{GS} to obtain all irreducible finite-dimensional
$\osp$-characters. Brundan and Stroppel \cite{BrS} also provided
another approach to the main results of \cite{Br} and
independently proved a special case of the super duality
conjecture in type $A$ as formulated in \cite{CWZ}. All these
approaches have brought new and different insights into super
representation theory. Our super duality approach has the
advantages of explaining the connection with classical Lie
algebras and their Kazhdan-Lusztig polynomials, covering
infinite-dimensional irreducible characters, and being extendable
to general Kac-Moody Lie superalgebras.

A list of symbols is added at the end of the paper to facilitate the reading.
\subsection{}

Let us end the Introduction with some remarks on the interrelations among the three
dualities.

The $(\gl(d), \gl(n))$-Howe duality is equivalent to Schur duality. It follows from the
Schur-Sergeev duality that the characters for irreducible polynomial $\glmn$-modules are
given by the so-called hook Schur functions. On the other hand, the irreducible character
formulas for Lie superalgebras of types $\gl$ or $\osp$ obtained from Howe duality can be
expressed in terms of infinite classical Weyl groups. The appearance of hook Schur
functions and infinite Weyl groups in these formulas are conceptually explained from the
viewpoint of super duality.

Super duality can be informally interpreted as a categorification of the standard
involution on the ring of symmetric functions. It is well known that the ring of
symmetric functions in infinitely many variables admits symmetries which are not observed
in finitely many variables. Super duality is formulated precisely at the infinite
rank limit. On the level of combinatorial parameterizations of highest weights, super
duality manifests itself through (variation of) the conjugate of partitions.


Partly due to the time constraint of the lectures, we have left out many interesting
topics on super representation theory. We refer to \cite{BL, dJ} (and more recently
\cite{SZ}) for finite-dimensional irreducible characters of atypicality one, to
\cite{BKN, DS, Ma, Mu, Pe, PS} for geometric approaches, to \cite{Br2, CWZ2} for further
development of the Fock space approach of Brundan for the queer Lie superalgebra $\mf{q}
(n)$ and for $\osp(2|2n)$, to \cite{CK, CKW, CZ1, Ger, San, Sva, Zou} for some
cohomological aspects, to \cite{BrK, SW, WZ} for prime characteristic, to \cite{dHKT, Su}
for related combinatorial structures; also see \cite{Gor, KW, Naz} for additional
work on Lie superalgebras.

\medskip

{\bf Acknowledgment.} This paper is a modified and expanded written account of the $8$
lectures given by the second author at the summer school of East China Normal University
(ECNU), Shanghai, July 2009. We are grateful to Ngau Lam for his collaboration and
insight. We thank Bin Shu at ECNU for hospitality and an enjoyable summer school.

\section{Lie superalgebra  ABC}
\label{sec:abc}

\subsection{}

A vector superspace $V$ is understood as a  $\Z_2$-graded vector space $V =\ev V \oplus
\od V$. An element $a \in  V_i$ has parity $|a|=i$, and an element in $\ev V$
(respectively, $\od V$) is called even (respectively, odd).

\begin{definition}
A Lie superalgebra is a vector superspace $\ga =\ev \ga \oplus\od \ga$ equipped with a
bilinear bracket operation $[.,.]$ satisfying $[\ga_i, \ga_j] \subset \ga_{i+j}$,
$i,j\in\Z_2$, and the following two axioms: for $\Z_2$-homogeneous $a, b, c \in \ga$,
\begin{enumerate}
\item (Skew-supersymmetry) $[a, b] =- (-1)^{|a|\cdot|b|} [b,a].$

\item (Super Jacobi identity) $[a,[b,c]] =[[a,b],c] +(-1)^{|a|\cdot|b|}[b,[a,c]].$
\end{enumerate}
\end{definition}

\begin{remark}
\begin{enumerate}
\item For a Lie superalgebra $\ga =\ev \ga \oplus\od \ga$, $\ev \ga$ is a Lie algebra
    and $\od \ga$ is a $\ev \ga$-module under the adjoint action.

\item(Sign Rule) As explained in Manin's book \cite{Ma}, there is a general heuristic
    sign rule for superalgebras as follows. {\em If in some formula for usual algebra
    there are monomials with interchanged terms, then in the corresponding formula for
    superalgebra every interchange of neighboring terms, say $a$ and $b$, is
    accompanied by the multiplication of the monomials by the factor
    $(-1)^{|a|\cdot|b|}$.} This is already manifest in the definition of Lie
    superalgebra and will persist throughout the paper.
\end{enumerate}
\end{remark}

\begin{example} \label{ex:super}
\begin{enumerate}
\item Let $A =\ev A \oplus \od A$ be an associative superalgebra
    (i.e.~$\Z_2$-graded). Then $(A, [.,.])$ is a Lie superalgebra, where for
    homogeneous elements $a,b \in A$, we define
$$
{[}a, b] =ab - (-1)^{|a|\cdot|b|} ba.
$$

\item A Lie superalgebra $\ga$ with $\od \ga =0$ is just a usual Lie algebra. A Lie
    superalgebra $\ga$ with purely odd part (i.e. $\ev \ga =0$) has to be {\em
    abelian}, i.e.~$[\ga,\ga]=0$.
\end{enumerate}
\end{example}

\subsection{Lie superalgebras of type $A$ and the supertrace}

Let $V =\ev V \oplus \od V$ be a vector superspace. Then $\End (V)$ is naturally an
associative superalgebra. The Lie superalgebra $\gl (V) :=(\End (V), [.,.])$ from
Example~\ref{ex:super}~(1) is called a {\em general linear Lie superalgebra}. If $\ev V
=\C^m$ and $\od V =\C^n$, we denote $V$ by $\C^{m|n}$, and $\gl (V)$ by $\gl (m|n)$. Note
that both $\gl(m|0) \cong \gl(0|m)$ are isomorphic to the usual Lie algebra $\gl(m)$.

The Lie superalgebra $\glmn$ consists of block matrices of size $m|n$:
\begin{equation} \label{matrix}
g =
\begin{pmatrix}
a & b\\
c & d
\end{pmatrix}.
\end{equation}
Throughout the paper, we choose to parameterize the rows and columns of the matrices by
the set
$$
I(m|n) =\{\ov 1, \ldots, \ov m; 1, \ldots, n\}
$$
with a total order
\begin{equation}  \label{eq:order}
\ov 1 < \ldots < \ov m  < 0 <  1 < \ldots < n
\end{equation}
(where $0$ is inserted
for later convenience). Its even subalgebra consists of matrices of the form
\begin{equation*}
\begin{pmatrix}
a & 0\\
0 & d
\end{pmatrix}
\end{equation*}
and is isomorphic to $\gl(m) \oplus \gl(n)$.

\begin{example}  \label{gl(1|1)}
For $\ga =\gl(1|1)$, let
\begin{equation*}
e= \begin{pmatrix}
0 & 1\\
0 & 0
\end{pmatrix},
 \qquad
f=
\begin{pmatrix}
0 & 0\\
1 & 0
\end{pmatrix},
 \qquad
h_1=
\begin{pmatrix}
1 & 0\\
0 & 0
\end{pmatrix},
 \qquad
h_2=
\begin{pmatrix}
0 & 0\\
0 & 1
\end{pmatrix}.
\end{equation*}
Then
\begin{equation*}
[e,f] =h_1 +h_2  \text{ (the identity matrix)}.
\end{equation*}
\end{example}

The {\em supertrace}, denoted by $\str$, of \eqref{matrix} is defined to be
$$
\str (g) = \tr(a) -\tr(d).
$$
The {\em special linear Lie superalgebra} is
$$
\sgl(m|n) =\{x \in \glmn \mid \str (x) =0\}.
$$
The definitions of supertrace and of the Lie superalgebra $\sgl$ are justified by the
following.

\begin{xca}
Show that $\sgl(m|n) =[\glmn, \glmn]$ and in particular $\sgl(m|n)$ is a Lie subalgebra
of $\glmn.$
\end{xca}
The notion of simple Lie superalgebras is defined in the same way as for Lie algebras.
We note  that $\sgl( n|n)$ is not a simple Lie superalgebra, as it contains a nontrivial
center $\C I_{2n}$.

\subsection{The bilinear form}

Let $\h$ denote the Cartan subalgebra of $\glmn$ consisting of all diagonal matrices.
Note that $\h$ is an even subalgebra of $\glmn$.

Let $E_{ij}$, for $i,j \in I(m|n)$, denote the standard basis for $\glmn$. We define a
bilinear form $(\cdot,\cdot)$ on $\ga$ by letting
$$
(a,b) =\str (ab), \quad a,b \in \ga.
$$
This restricts to a nondegenerate symmetric bilinear form on $\h$: for $i,j \in I(m|n)$,
\begin{equation*}
(E_{ii}, E_{jj}) =\left\{
\begin{array}{rl}
 1 & \text{ if } \ov 1  \le i=j \le \ov m,\\
 -1 & \text{ if } 1\le i=j \le n,\\
 0  & \text{ if }   i \neq j.
\end{array}
\right.
\end{equation*}
Denote by $\{\delta_i,\epsilon_j\}_{i,j}$ the basis of $\h^*$ dual to $\{E_{\ov i \,\ov
i}, E_{jj}\}_{i,j}$, where $1\le i \le m$ and $1\le j \le n$. Under the bilinear form
$(\cdot,\cdot)$, we have the identification $\delta_i =(E_{\ov i,\ov i}, \cdot)$ and
$\epsilon_j =-(E_{jj},\cdot)$. Whenever it is convenient we also use the notation
\begin{equation} \label{ep=delta}
\ep_{\ov i} :=\delta_i, \quad \text{ for }1\le i \le m.
\end{equation}

The form $(\cdot,\cdot)$ on $\h$ induces a non-degenerate bilinear from on $\h^*$, which
will be denoted by the same notation, as follows: for $i,j \in I(m|n)$,
\begin{equation} \label{form}
\begin{split}
(\ep_i, \ep_j) =\left\{
\begin{array}{rl}
 1 & \text{ if } \ov 1 \le i=j \le \ov m,\\
 -1 & \text{ if } 1\le i=j \le n,\\
 0  & \text{ if }   i \neq j.
\end{array}
\right.\end{split}
\end{equation}
\subsection{The root system}

For the Lie superalgebra $\glmn$, we define the root space decomposition, a root system
$\Phi$, a set $\Phi^+$ (respectively, $\Phi^-$) of positive (respectively, negative)
roots, a set $\Pi$ of simple roots (in $\Phi^+$), etc. As this can be done in the same
way as for semisimple Lie algebras or $\gl (m)$, we will merely write down the statements
for later use.

Now let us make the super phenomenon explicit. A root $\alpha$ is {\em even} if
$\ga_\alpha \subset \ev \ga$, and it is {\em odd} if $\ga_\alpha \subset \od \ga$. Denote
by $\ev \Phi$ (respectively, $\od \Phi$) the set of all even (respectively, odd) roots in
$\Phi$. Denote
\begin{align*}
\Phi^\pm_i=\Phi_i\cap\Phi^\pm,\quad \Pi_i =\Phi_i \cap \Pi,\quad i\in\Z_2.
\end{align*}

With respect to the Cartan subalgebra $\h$ the Lie superalgebra $\glmn$ admits a root
space decomposition:
$$
\ga = \h \oplus \bigoplus_{\alpha \in \Phi} \ga_\alpha,
$$
with a root system
$$
\Phi = \{ \ep_i -\ep_j \mid i, j \in I(m|n), i \ne j\}.
$$
The standard set of simple roots is taken to be $\Pi =\ev \Pi \cup \od \Pi$, where
$$
\ev \Pi =\{\ep_{\ov i} -\ep_{\overline{i+1}}\}_{1\le i \le {m-1}}
\cup \{\ep_i -\ep_{i+1}\}_{1\le i \le n-1}, \quad \od \Pi =
\{\ep_{\ov m} -\ep_1\},
$$
and the associated standard set of positive roots is
$$
\Phi^+ = \{\ep_i -\ep_j \mid i , j \in I(m|n), i<j\},
$$
where the odd roots are $\ep_i -\ep_j$ with indices $i<0<j$. Clearly, $\ga_{\ep_i -\ep_j}
= \C E_{ij}$.
It follows by \eqref{form} that
$$
(\delta_i -\ep_j, \delta_i -\ep_j) =0,
$$
for all the odd roots $\delta_i -\ep_j$, where $1\le i\le m, 1 \le j\le n$. An odd root
$\alpha$ with $(\alpha, \alpha) =0$ is called {\em isotropic}.
The standard Dynkin diagram is:
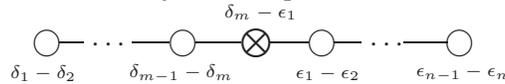
\begin{center}
\hskip -3cm \setlength{\unitlength}{0.16in}
\begin{picture}(24,3)
\put(8.4,2){\makebox(0,0)[c]{$\bigcirc$}} \put(12.85,2){\makebox(0,0)[c]{$\bigcirc$}}
\put(15.25,2){\makebox(0,0)[c]{$\bigotimes$}} \put(17.4,2){\makebox(0,0)[c]{$\bigcirc$}}
\put(21.9,2){\makebox(0,0)[c]{$\bigcirc$}} \put(8.82,2){\line(1,0){0.8}}
\put(11.2,2){\line(1,0){1.2}} \put(13.28,2){\line(1,0){1.45}}
\put(15.7,2){\line(1,0){1.25}} \put(17.8,2){\line(1,0){0.9}}
\put(20.1,2){\line(1,0){1.4}} \put(10.5,1.95){\makebox(0,0)[c]{$\cdots$}}
\put(19.6,1.95){\makebox(0,0)[c]{$\cdots$}}
\put(8.3,1){\makebox(0,0)[c]{\tiny $\delta_1 -\delta_2$}}
\put(12.8,1){\makebox(0,0)[c]{\tiny $\delta_{m-1} -\delta_{m}$}}
\put(15.4,3){\makebox(0,0)[c]{\tiny $\delta_{m}  -\ep_{1}$}}
\put(17.6,1){\makebox(0,0)[c]{\tiny $\ep_{1} -\ep_{2}$}}
\put(22,1){\makebox(0,0)[c]{\tiny $\ep_{n-1} -\ep_{n}$}}
\end{picture}
\end{center}
where we have used $\bigotimes$ to denote an isotropic odd simple root.

\begin{remark}
The notion of root systems and Dynkin diagrams makes sense for all the basic classical
Lie superalgebras, which consist of $\glmn, \mathfrak{sl}(m|n), \osp(m|2n)$ and three exceptional ones
(besides the simple Lie algebras).
\end{remark}

\subsection{Non-conjugate Borel subalgebras and $\ep\delta$-sequences}
\label{sec:ep delta}

Recall that the bilinear form on the real subspace $\h_\R^*$
spanned by the $\ep_i$'s is not positive definite (due to the
supertrace), and moreover, there exist isotropic odd roots.

Another distinguished feature of Lie superalgebras  is the existence of non-conjugate
Borel subalgebras or non-isomorphic Dynkin diagrams (under the Weyl group action).

\begin{lemma} \label{odd ref}
Let $\ga$ be a Lie superalgebra with triangular decomposition $\ga
=\n^- \oplus \h \oplus \n^+$, which corresponds to the root system
$\Phi =\Phi^+ \cup \Phi^-$. Let $\alpha$ be an odd isotropic
simple root. Let $\mf b =\h +\n^+$. Then, $\Phi(\alpha)^+ : =
(\Phi^+ \backslash \{\alpha\}) \cup \{-\alpha\}$ is a new
    system of positive roots, whose corresponding set of simple roots
    is
$$
\Pi(\alpha) =\{\beta \in \Pi \mid (\beta, \alpha) =0, \beta \neq
\alpha\} \cup \{\beta +\alpha \mid \beta \in \Pi, (\beta, \alpha)
\neq 0\} \cup \{-\alpha\}.
$$
\end{lemma}
The new Borel subalgebra corresponding to $\Pi(\alpha)$ will be denoted by $\mf
b(\alpha)$.
\begin{proof}
Follows from a straightforward verification.
\end{proof}
The process of obtaining $\Pi(\alpha)$ from $\Pi$ above will be
referred to as an {\em odd reflection}, and will be denoted by
$r_\alpha$, in accordance with the usual notion of real
reflections.

\begin{example}  \label{gl12}
Associated to $\gl(1|2)$, we have $\ev \Phi = \{ \pm (\ep_1
-\ep_2)\}$, and $\od \Phi =\{ \pm (\delta_1 -\ep_1), \pm (\delta_1
-\ep_2)\}.$ There are $6$ sets of simple roots, that are related
by the real and odd reflections as follows. There are three
conjugacy classes of Borel subalgebras, and each vertical pair
corresponds to such a conjugacy class.

\vspace{.3cm}

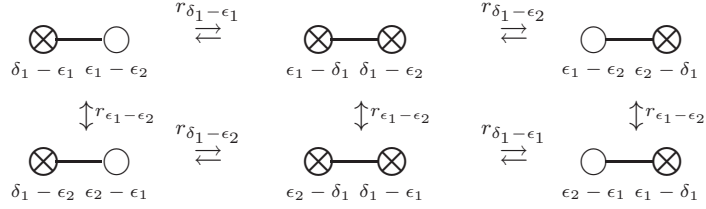
\begin{center}
\hskip -1cm \setlength{\unitlength}{0.16in}
\begin{picture}(22,7)
\put(1.6,6){\makebox(0,0)[c]{$\bigotimes$}} \put(4,6){\makebox(0,0)[c]{$\bigcirc$}}
\put(2,6){\line(1,0){1.5}} \put(1.6,5){\makebox(0,0)[c]{\tiny $\delta_1-\epsilon_1$}}
\put(4,5){\makebox(0,0)[c]{\tiny $\epsilon_1-\epsilon_2$}}
\put(7,6.5){\makebox(0,0)[c]{$\stackrel{r_{\delta_1-\epsilon_1}}{\rightleftarrows}$}}
\put(10.6,6){\makebox(0,0)[c]{$\bigotimes$}} \put(13,6){\makebox(0,0)[c]{$\bigotimes$}}
\put(11,6){\line(1,0){1.5}} \put(10.6,5){\makebox(0,0)[c]{\tiny $\epsilon_1-\delta_1$}}
\put(13,5){\makebox(0,0)[c]{\tiny $\delta_1-\epsilon_2$}}
\put(17,6.5){\makebox(0,0)[c]{$\stackrel{r_{\delta_1-\epsilon_2}}{\rightleftarrows}$}}
\put(19.6,6){\makebox(0,0)[c]{$\bigcirc$}} \put(22,6){\makebox(0,0)[c]{$\bigotimes$}}
\put(20,6){\line(1,0){1.5}} \put(19.6,5){\makebox(0,0)[c]{\tiny $\epsilon_1-\epsilon_2$}}
\put(22,5){\makebox(0,0)[c]{\tiny $\epsilon_2-\delta_1$}}
\put(1.6,2){\makebox(0,0)[c]{$\bigotimes$}} \put(4,2){\makebox(0,0)[c]{$\bigcirc$}}
\put(2,2){\line(1,0){1.5}} \put(1.6,1){\makebox(0,0)[c]{\tiny $\delta_1-\epsilon_2$}}
\put(4,1){\makebox(0,0)[c]{\tiny $\epsilon_2-\epsilon_1$}}
\put(7,2.5){\makebox(0,0)[c]{$\stackrel{r_{\delta_1-\epsilon_2}}{\rightleftarrows}$}}
\put(10.6,2){\makebox(0,0)[c]{$\bigotimes$}} \put(13,2){\makebox(0,0)[c]{$\bigotimes$}}
\put(11,2){\line(1,0){1.5}} \put(10.6,1){\makebox(0,0)[c]{\tiny $\epsilon_2-\delta_1$}}
\put(13,1){\makebox(0,0)[c]{\tiny $\delta_1-\epsilon_1$}}
\put(17,2.5){\makebox(0,0)[c]{$\stackrel{r_{\delta_1-\epsilon_1}}{\rightleftarrows}$}}
\put(19.6,2){\makebox(0,0)[c]{$\bigcirc$}} \put(22,2){\makebox(0,0)[c]{$\bigotimes$}}
\put(20,2){\line(1,0){1.5}} \put(19.6,1){\makebox(0,0)[c]{\tiny $\epsilon_2-\epsilon_1$}}
\put(22,1){\makebox(0,0)[c]{\tiny $\epsilon_1-\delta_1$}}
\put(3,3.5){\makebox(0,0)[c]{$\updownarrow$}}
\put(4.3,3.5){\makebox(0,0)[c]{\tiny{$r_{\epsilon_1-\epsilon_2}$}}}
\put(12,3.5){\makebox(0,0)[c]{$\updownarrow$}}
\put(13.3,3.5){\makebox(0,0)[c]{\tiny{$r_{\epsilon_1-\epsilon_2}$}}}
\put(21,3.5){\makebox(0,0)[c]{$\updownarrow$}}
\put(22.3,3.5){\makebox(0,0)[c]{\tiny{$r_{\epsilon_1-\epsilon_2}$}}}
\end{picture}
\end{center}
\end{example}
One convenient way to parameterize the conjugacy classes of Borel subalgebras of $\glmn$
is via the notion of $\ep\delta$-sequences. Keeping \eqref{ep=delta} in mind,
we list the simple roots associated to a given
Borel subalgebra $\mf b$ in order as $\ep_{i_1} -\ep_{i_2}, \ep_{i_2} -\ep_{i_3}, \ldots,
\ep_{i_{m+n-1}} -\ep_{i_{m+n}}$, where $\{i_1, i_2, \ldots, i_{m+n}\} =I(m|n)$. Switching
the ordered sequence $\ep_{i_1}\ep_{i_2}\ldots \ep_{i_{m+n}}$  to the
$\ep\delta$-notation by \eqref{ep=delta} and then dropping the indices give us the {\em
$\ep\delta$-sequence} associated to $\mf b$. Note that the total number of $\delta$'s
(respectively, $\ep$'s) is $m$ (respectively, $n$).

For example, the three conjugacy classes of Borels for $\gl(1|2)$ above correspond to the
three sequences $\delta \ep\ep$, $\ep\delta\ep$, $\ep\ep\delta$, respectively. In more
detail, the first sequence $\delta \ep\ep$ is obtained by removing the indices of
$\delta_1 \ep_1  \ep_2$ (read off from the upper-left diagram above) or $\delta_1 \ep_2
\ep_1$ (from the lower-left diagram above). Also the standard Borel of $\glmn$
corresponds to the sequence $\underbrace{\delta\cdots \delta}_m \underbrace{\ep \cdots
\ep}_n$ while the opposite Borel to the standard one corresponds to $\underbrace{\ep
\cdots \ep}_n \underbrace{\delta \cdots \delta}_m$.

\begin{xca} Let $\Phi$ be the roots of $\glmn$ with respect to the Cartan subalgebra
$\h$.  Prove that the sets of simple roots in $\Phi$ are in one-to-one correspondence
with the $\ep\delta$-sequences with $m$ $\delta$'s and $n$ $\ep$'s in total.  In
particular, there are $\binom{m+n}{m}$ conjugacy classes of Borel subalgebras.
\end{xca}

\subsection{}
Let $B$ be a non-degenerate even supersymmetric bilinear form on a vector superspace $V
=\ev V \oplus \od V$. Here $B$ is {\em even} if $B(V_i, V_j) =0$ unless $i=j \in \Z_2$,
and $B$ is {\em supersymmetric} if $B|_{\ev V \times \ev V}$ is symmetric while $B|_{\od
V \times \od V}$ is skew-symmetric (and hence $\dim \od V$ is necessarily even).

For $s\in \Z_2$, let
\begin{eqnarray*}
\osp (V)_s &=&\{g \in \gl (V)_s \mid B(g(x), y) =-(-1)^{s\cdot
|x|}
B(x, g(y)), \forall x, y \in V \},  \\
\osp(V) &=& \ev \osp(V) \oplus \od \osp(V).
\end{eqnarray*}
One checks that $\osp(V)$ is a Lie superalgebra, whose even subalgebra is isomorphic to
$\mf{so} (\ev V) \oplus \mf{sp} (\od V).$ When $V = \C^{\ell|2n}$, we write $\osp (V)
=\osp (\ell |2n)$.

\begin{remark}
One can also define the Lie superalgebra $\spo(V)$ as the subalgebra of $\gl(V)$ which
preserves a non-degenerate skew-supersymmetric bilinear form on $V$ (here $\dim \ev V$
has to be even). When $V = \C^{2n|\ell}$, we write $\spo (V) =\spo (2n|\ell)$.
\end{remark}

\begin{xca}
Show that Lie superalgebras $\osp (\ell |2n)$ and $\spo (2n|\ell)$ are isomorphic.
\end{xca}

\subsection{}
Define the {\em super transpose} $\text{st}$ as follows: for a matrix in the block form
\eqref{matrix}, we let
\begin{equation*}
\begin{pmatrix}
a & b\\
c & d
\end{pmatrix}^{\text{st}}
 =
\begin{pmatrix}
a^t & c^t\\
-b^t & d^t
\end{pmatrix},
\end{equation*}
where $x^t$ denotes the usual transpose of the matrix $x$.

Define the $(2n+2m+1) \times (2n+2m+1)$ matrix in the $(n|n|m|m|1)$-block form
\begin{eqnarray} \label{symform}
\mathfrak J_{2n|2m+1} :=
    \begin{pmatrix}
        0  &  I_n  & 0   & 0  & 0  \\
     -I_n  &  0    & 0   & 0  & 0  \\
        0  &  0    & 0   &I_m & 0  \\
        0  &  0    & I_m & 0  & 0  \\
        0  &  0    & 0   & 0  & 1
    \end{pmatrix},
\end{eqnarray}
where $I_n$ is the $n \times n$ identity matrix. Let $\mathfrak J_{2n|2m}$ denote the
$(2n+2m) \times (2n+2m)$ matrix obtained from $\mathfrak J_{2n|2m+1}$ by deleting the
last row and column. Then, for $\ell =2m$ or $2m+1$, by definition $ \spo(2n|\ell)$ is the subalgebra of
$\gl(2n|\ell)$ which preserves the bilinear form on $\C^{2n|\ell}$ with matrix $\mathfrak
J_{2n|\ell}$ relative to the standard basis of $\C^{2n|\ell}$, and hence
$$ \spo(2n|\ell) = \{g \in \gl (2n|\ell) \mid g^{st} \mathfrak
J_{2n|\ell} + \mathfrak J_{2n|\ell} \, g =0 \}.
$$
By a direct computation, $\spo(2n|2m+1)$ consists of the $(2n+2m+1) \times (2n+2m+1)$
matrices of the following $(n|n|m|m|1)$-block form
\begin{eqnarray} \label{matrixSPO}
    \begin{pmatrix}
    d      &   e     & y_1^t  & x_1^t &  z_1^t       \\
    f      &  -d^t   & -y^t   & - x^t &  -z^t  \\
    x      & x_1     &  a     &  b    &  -v^t  \\
    y      & y_1     &  c     & -a^t  &  -u^t   \\
    z      & z_1    & u      &  v    &    0
    \end{pmatrix},
\quad b, c \text{ skew-symmetric}, e, f \text{
    symmetric}.
\end{eqnarray}
The Lie superalgebra $\spo(2n|2m)$ consists of matrices
\eqref{matrixSPO} with the last row and column removed. Here and
below, the rows and columns of the matrices $\mathfrak
J_{2n|\ell}$ and (\ref{matrixSPO}) (or its modification) are
indexed by the finite set $I(2n|\ell)$.

The standard Dynkin diagrams of $\spo(2n|2m+1)$  and $\spo(2n|2m)$ are given respectively
as follows:

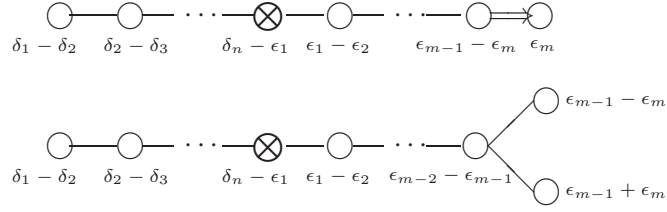
\begin{center}
\hskip -3cm \setlength{\unitlength}{0.16in}
\begin{picture}(24,3)
\put(5.7,2){\makebox(0,0)[c]{$\bigcirc$}} \put(8,2){\makebox(0,0)[c]{$\bigcirc$}}
\put(10.4,2){\makebox(0,0)[c]{$\cdots$}} \put(12.5,1.95){\makebox(0,0)[c]{$\bigotimes$}}
\put(14.85,2){\makebox(0,0)[c]{$\bigcirc$}} \put(17.25,2){\makebox(0,0)[c]{$\cdots$}}
\put(19.4,2){\makebox(0,0)[c]{$\bigcirc$}} \put(21.4,2){\makebox(0,0)[c]{$\bigcirc$}}
\put(6,2){\line(1,0){1.55}} \put(8.4,2){\line(1,0){1}} \put(11,2){\line(1,0){1}}
\put(13.1,2){\line(1,0){1.2}} \put(15.28,2){\line(1,0){1}} \put(17.7,2){\line(1,0){1.1}}
\put(19.7,1.75){$\Longrightarrow$} \put(5.2,1){\makebox(0,0)[c]{\tiny
$\delta_1-\delta_{2}$}} \put(8.2,1){\makebox(0,0)[c]{\tiny $\delta_{2}-\delta_{3}$}}
\put(12.1,1){\makebox(0,0)[c]{\tiny $\delta_n-\epsilon_1$}}
\put(14.8,1){\makebox(0,0)[c]{\tiny $\epsilon_1-\epsilon_2$}}
\put(19,1){\makebox(0,0)[c]{\tiny $\epsilon_{m-1}-\epsilon_m$}}
\put(21.5,1){\makebox(0,0)[c]{\tiny $\epsilon_{m}$}}
\end{picture}
\end{center}
\begin{center}
\hskip -3cm \setlength{\unitlength}{0.16in}
\begin{picture}(24,4)
\put(5.7,2){\makebox(0,0)[c]{$\bigcirc$}} \put(8,2){\makebox(0,0)[c]{$\bigcirc$}}
\put(10.4,2){\makebox(0,0)[c]{$\cdots$}} \put(12.5,1.95){\makebox(0,0)[c]{$\bigotimes$}}
\put(14.85,2){\makebox(0,0)[c]{$\bigcirc$}} \put(17.25,2){\makebox(0,0)[c]{$\cdots$}}
\put(19.3,2){\makebox(0,0)[c]{$\bigcirc$}} \put(21.6,3.5){\makebox(0,0)[c]{$\bigcirc$}}
\put(21.6,0.5){\makebox(0,0)[c]{$\bigcirc$}} \put(6,2){\line(1,0){1.55}}
\put(8.4,2){\line(1,0){1}} \put(11,2){\line(1,0){1}} \put(13.1,2){\line(1,0){1.2}}
\put(15.28,2){\line(1,0){1}} \put(17.7,2){\line(1,0){1.1}} \put(19.7,2){\line(1,1){1.5}}
\put(19.7,2){\line(1,-1){1.5}} \put(5.2,1){\makebox(0,0)[c]{\tiny $\delta_1-\delta_{2}$}}
\put(8.2,1){\makebox(0,0)[c]{\tiny $\delta_{2}-\delta_{3}$}}
\put(12.1,1){\makebox(0,0)[c]{\tiny $\delta_n-\epsilon_1$}}
\put(14.8,1){\makebox(0,0)[c]{\tiny $\epsilon_1-\epsilon_2$}}
\put(18.5,1){\makebox(0,0)[c]{\tiny $\epsilon_{m-2}-\epsilon_{m-1}$}}
\put(23.9,3.5){\makebox(0,0)[c]{\tiny $\epsilon_{m-1}-\epsilon_{m}$}}
\put(23.9,0.5){\makebox(0,0)[c]{\tiny $\epsilon_{m-1}+\epsilon_{m}$}}
\end{picture}
\end{center}

\begin{example}
The Lie superalgebra $\ga =\osp(1|2)$ has even subalgebra $\ev \ga \cong \sgl(2)
=\C\langle e, h, f\rangle$ and $\od \ga$ isomorphic to the $2$-dimensional natural
$\sgl(2)$-module $\C \langle E, F \rangle$. Moreoever, $[E,E] =2 e, [F,F] =-2 f, [E, F]
=h$.

The simple root consists of a (unique) odd non-isotropic root $\delta$, twice of which is
an even root. The Dynkin diagram of $\osp(1|2)$ is denoted by
\makebox(8,6){\small\CircleSolid} (in order to distinguish from an odd simple isotropic
root $\bigotimes$).
\end{example}

\begin{xca}\label{xca:osp1} Prove the following identities in $U(\osp(1|2))$ ($n\in\Z_+$):
\begin{itemize}
\item[(i)] $[E,F^{2n}]=-nF^{2n-1}$,
\item[(ii)] $[E,F^{2n+1}]=F^{2n}(h-n)$.
\end{itemize}
\end{xca}

\begin{xca}\label{xca:osp2}
Use Exercise \ref{xca:osp1} to show that the finite-dimensional irreducible
representations of $\osp(1|2)$ are parameterized by the set of highest weights
$\{n\delta|\n\in\Z_+\}$. Denoting the irreducible module corresponding to $n\delta$ by
$L(n\delta)$, show that
\begin{align*}
\text{ch}L(n\delta)=\frac{e^{(n+\hf)\delta}-e^{-(n+\hf)\delta}}{e^{\hf\delta}-e^{-\hf\delta}}.
\end{align*}
\end{xca}

\begin{xca} Consider the element
\begin{align*}
\Omega:=2h^2+2h+FE+4fe\in U(\osp(1|2)).
\end{align*}
Use Exercise \ref{xca:osp1} to prove that $[F,\Omega]=[E,\Omega]=0$, and hence $\Omega$ is
in the center of $U(\osp(1|2))$. Conclude from Exercise \ref{xca:osp2} that $\Omega$ acts
as different scalars on different finite-dimensional irreducible representations and
hence that every finite-dimensional $\osp(1|2)$-module is completely reducible.
\end{xca}

\section{Finite-dimensional modules of Lie superalgebras}
\label{sec:typical}

\subsection{PBW theorem}

The {\em universal enveloping algebra} $U(\ga)$ of a Lie superalgebra $\ga=\ev \ga \oplus
\od \ga$ is an associative superalgebra characterized by a universal property exactly as for
Lie algebras.

Assume that $\{x_1,\ldots, x_r\}$ is a basis for $\ev \ga$ and $\{y_1,\ldots, y_s\}$ is a
basis for $\od \ga$. Then the universal enveloping algebra $U(\ga)$ admits a PBW basis
\begin{equation} \label{pbw}
x_1^{a_1} \cdots x_r^{a_r}  y_1^{b_1} \cdots y_s^{b_s},\qquad
\forall a_i \in \Z_{\ge 0}, \forall b_j \in \{0,1\}.
\end{equation}
Alternatively, if we define the standard PBW filtration $\{F^d U(\ga)\}$ on $U(\ga)$ by
letting $F^d U(\ga)$ be the span of elements \eqref{pbw} with $\Sigma_i a_i +\Sigma_j b_j
\le d$, then we have the following isomorphism for the associated graded of $U(\ga)$ in
terms of the symmetric algebra $S(\ev \ga)$ and exterior algebra $\wedge (\od \ga)$:

$$
\gr^F U(\ga) \cong S(\ev \ga) \otimes \wedge (\od \ga).
$$

\begin{example}
 Let $\ga =\od \ga$ be a purely odd Lie superalgebra. Then its
universal enveloping algebra is isomorphic to the exterior algebra $\wedge (\od
\ga)$.
\end{example}

\subsection{Representations of $\gl(1|1)$} \label{sec:1|1}

Recall the basis $\{e, h_1, h_2, f\}$ for $\ga =\gl(1|1)$ from
Example~\ref{gl(1|1)}. Note that the even subalgebra
$\ev{\gl(1|1)} =\gl(1) \oplus \gl(1)$ coincides with the Cartan
subalgebra $\h$, and the Weyl group is trivial in this case.

Given $\la =(\la_1, \la_2) \in \C^2$, we denote by $\C_\la :=\C \, v_\la$ the
one-dimensional $\h$-module by letting $h_i v_\la = \la_i v_\la$, for $i=1,2$. Regarding
$\C_\la$ as a module over the Borel subalgebra $\h +\C e$, on which $e$ acts trivially,
we define the Verma modules (which coincide with Kac modules defined below for
$\gl(1|1)$) over $\ga$:
$$
K(\la) =U(\ga) \otimes_{U(\h +\C e)} \C_\la.
$$
By the PBW theorem $K(\la)$ can be identified with the following two-dimensional space
(where by abuse of notation $1 \otimes v_\la$ is denoted by $v_\la$):
$$
K(\la)  = \C \langle v_\la, f v_\la \rangle.
$$
The $\ga$-module $K(\la)$ has a unique simple quotient, denoted by $L(\la)$.

The $\ga$-module $K(\la)$ is simple if and only if $\la_1 \neq
-\la_2$. This follows from
$$
ef v_\la =(h_1 +h_2) v_\la -f e v_\la = (\la_1 +\la_2) v_\la.
$$

If $\la_1 \neq -\la_2$, then $K(\la_1, \la_2)$ is the unique
simple object in its block (in the category of finite-dimensional
$\gl(1|1)$-modules). This block is semisimple.

For $a \in \C$, we have a non-split short exact sequence of $\ga$-modules:
$$
0 \to L(a-1, 1-a) \to K(a, -a) \to L(a, -a) \to 0.
$$
The block containing $L(a, -a)$ is not semisimple, and it contains
infinitely many simple objects $L(b,-b)$ for $b \in a +\Z$ (the
underlying algebra can be described in terms of the
$A_\infty$-quiver).

\subsection{Finite dimensional simple $\glmn$-modules}

Let $\ga =\glmn$ in this subsection.

Let $\n^+$ (and respectively, $\n^-$) be the subalgebra of strictly upper (and
respectively, lower) triangular matrices of $\glmn$ with respect to the standard basis.
Then we have the triangular decomposition
$$
\ga =\n^- \oplus \h \oplus \n^+.
$$
The even subalgebra admits a compatible triangular decomposition
$$
\ev \ga =\ev \n^- \oplus \h \oplus \ev \n^+,
$$
where $\ev \n^\pm =\ev \ga \cap \n^\pm$.

Moreover, the Lie superalgebra $\ga$ admits a $\Z$-grading
$$
\ga =\ga_{-1} \oplus \ev \ga \oplus \ga_1,
$$
where $\ga_{-1}$ (respectively, $\ga_1$) is spanned by all $E_{ij}$ with $i>0>j$
(respectively, $i<0<j$).  Note that this $\Z$-grading is compatible with the
$\Z_2$-grading, i.e., the degree zero subspace coincides with the $\Z_2$-degree zero
subspace. This is equivalent to the fact that $\ev \ga$ is a Levi subalgebra of $\ga$
(corresponding to the removal of the odd simple root from the standard Dynkin diagram of
$\ga$).

For $\la \in \h^*$, let $L(\la)$ (respectively, $L^0(\la)$) be the simple module of $\ga$
(respectively, $\ev \ga$) of highest weight $\la$. Define the {\em Kac module} over $\ga$
by
$$
K(\la) = U(\ga) \otimes_{U(\ev \ga \oplus \ga_1)} L^0(\la),
$$
which can be identified by the PBW theorem with
\begin{equation} \label{Kmodule}
K(\la) =\wedge (\ga_{-1}) \otimes L^0(\la).
\end{equation}

\begin{proposition}
There exists a surjective $\ga$-module homomorphism (unique up to a scalar multiple)
$K(\la) \twoheadrightarrow L(\la)$. Moreover, the following are equivalent:
\begin{enumerate}
\item $L(\la)$ is finite-dimensional.

\item $L^0(\la)$ is finite-dimensional.

\item $K(\la)$ is finite-dimensional.
\end{enumerate}
\end{proposition}

\begin{proof}
The surjective homomorphism follows from the fact that $K(\la)$ is a highest weight
$\ga$-module of highest weight $\la$.

$(1) \Rightarrow (2).$ Note that $L^0(\la)$ is a quotient of $L(\la)$ regarded as $\ev
\ga$-module.

$(2) \Rightarrow (3).$ Follows from \eqref{Kmodule}.

$(3) \Rightarrow (1).$ Follows from the surjective map $K(\la) \twoheadrightarrow
L(\la)$.
\end{proof}
Arguing as for Lie algebras, every finite-dimensional simple $\ga$-module is a highest
weight module $L(\la)$ for some $\la$, and moreover, $L(\la) \not \cong L(\mu)$
if $\la \neq \mu$. Hence, the above proposition gives a classification of
finite-dimensional simple $\ga$-modules.

\subsection{Typical irreducible characters}

Let $\ga =\glmn$ in this subsection.

We will denote by $P^+$ the set of weights $\la \in \h^*$ such that $L^0(\la)$ is
finite-dimensional. Recall that the character of $L^0(\la)$ for $\la \in P^+$ is given by
Weyl's formula
$$
\ch L^0(\la) =\frac{\sum_{\sigma \in W} \sgn (\sigma) e^{\sigma
(\la +\rho_0) -\rho_0}}{\prod_{\alpha \in \ev \Phi^+} (1- e^{-\alpha})}.
$$
Here $W$ ($\cong {\mathfrak S}_m \times {\mathfrak S}_n$) denotes the Weyl group of $\ev
\ga$, $\sgn$ denotes the sign representation of $W$, and $\rho_0 =\hf \sum_{\alpha \in
\ev \Phi^+} \alpha$.

Also denote by
$$
\rho_1 =\hf \sum_{\alpha \in \od \Phi^+} \alpha, \qquad \rho
=\rho_0 -\rho_1.
$$

\begin{lemma}\label{lem:000}
We have
\begin{enumerate}
\item $(\rho, \beta) =\hf (\beta, \beta), \forall \beta \in \Pi.$

\item $\sigma (\rho_1) =\rho_1, \forall \sigma \in W.$
\end{enumerate}
\end{lemma}

\begin{proof} Part 1 can be verified
directly. Part 2 follows from the fact that $\ga_1$ is preserved by the adjoint group
$\ev G$ associated to $\ev \ga$ under the adjoint action.
\end{proof}

\begin{corollary}
We have
\begin{equation} \label{typ ch}
\ch K(\la) =\frac{\prod_{\alpha \in \od \Phi^+} (1+
e^{-\alpha})}{\prod_{\alpha \in \ev \Phi^+} (1- e^{-\alpha})}
\sum_{\sigma \in W} \sgn (\sigma) e^{\sigma (\la +\rho) -\rho}.
\end{equation}
\end{corollary}

\begin{proof}
By \eqref{Kmodule}, $\ch K(\la) =\ch L^0(\la) \prod_{\alpha \in \od \Phi^+} (1+
e^{-\alpha})$. Now by Weyl's character formula for $L^0(\la)$, we have
$$
\ch K(\la) =\frac{\prod_{\alpha \in \od \Phi^+} (1+
e^{-\alpha})}{\prod_{\alpha \in \ev \Phi^+} (1- e^{-\alpha})}
\sum_{\sigma \in W} \sgn (\sigma) e^{\sigma (\la +\rho_0)
-\rho_0}.
$$
This is equivalent to the formula in (\ref{typ ch}) by the second part of Lemma
\ref{lem:000}.
\end{proof}
\begin{xca}
Show that
$$
 \rho_1 =\frac{n}2 \sum_{i=1}^{m} \delta_i - \frac{m}2 \sum_{j=1}^n \ep_j.
$$
Again it follows that  $\sigma (\rho_1) =\rho_1, \forall \sigma \in W$.
\end{xca}

\begin{definition} \label{def:typ}
A weight $\la \in \h^*$ is called {\em typical}, if $\prod_{\alpha \in \od \Phi^+} (\la
+\rho, \alpha) \neq 0$. It is called {\em atypical} if it is not typical.
\end{definition}

\begin{theorem} [Kac]  \label{typical ch}
Let $\la \in P^+$. Then $\la$ is typical if and only if $\ch L(\la)$ is given by the right hand
side of \eqref{typ ch}, if and only if $K(\la)$ is irreducible.
\end{theorem}
The example of $\gl(1|1)$ in Section~\ref{sec:1|1} fits well with this theorem.

\begin{proof}[Sketch of a proof] \cite{K2}
Let $v_\la \in K(\la)$ be a highest weight vector. Denote by $e_\alpha, f_\alpha$ the
generators of $\ga_{\pm \alpha}$, where $\alpha \in \Phi^+$.

Assume that $S$ is a nonzero $\ga$-submodule of $K(\la)$. One first shows easily that

(i) $\prod_{\alpha \in \od \Phi^+} f_{\alpha} v_\la \in S$.

It follows that

(ii) $ \prod_{\beta \in \od \Phi^+} e_{\beta} \prod_{\alpha \in \od \Phi^+} f_{\alpha}
v_\la \in S$.

Then one further shows that (up to a nonzero scalar multiple)

(iii)   $ \prod_{\beta \in \od \Phi^+} e_{\beta} \prod_{\alpha \in \od \Phi^+} f_{\alpha}
v_\la = \prod_{\alpha \in \od \Phi^+} (\la +\rho, \alpha) \ v_\la$.

Step (iii) uses the $W$-invariance of $\od \Phi^+$ and a degree counting argument among
others.

From (i) and (iii) it follows that if $\la$ is atypical, then $K(\la)$ is irreducible.

Now suppose that $K(\la)$ is irreducible. The $\text{ad } \ga_{{0}}$-invariance of $\ga_{\pm 1}$
implies that if $v_\la\in U(\ga)\prod_{\alpha\in\Phi^+_{\bar{1}}}f_\alpha v_\la$,
then $v_\la$ is a scalar multiple of $\prod_{\beta \in \od \Phi^+} e_{\beta}
\prod_{\alpha \in \od \Phi^+} f_{\alpha} v_\la$. Thus by (iii) $\la$ is typical.
\end{proof}
\subsection{Odd reflections}

As usual, we let $\{e_\beta, h_\beta,  f_\beta\}$ denote the Chevalley generators  for
$\beta \in \Phi^+$. The following simple and fundamental lemma for odd reflections has
been used by many authors (e.g.~\cite[Appendix]{LSS}, \cite[Lemma 0.3]{PS} and
\cite[Lemma 1.4]{KW}; also compare \cite[(2.12)]{DP} for an unconventional definition).
Recall the new Borel $\mf{b}(\alpha)$ for  a simple isotropic odd root $\alpha$ from
Lemma~\ref{odd ref}.

\begin{lemma}  \label{hwt odd}
Let $L$ be a simple $\ga$-module of $\mf b$-highest weight $\la$
and let $v$ be a $\mf b$-highest weight vector of $L$. Let
$\alpha$ be a simple isotropic odd root.
\begin{enumerate}
\item If $\langle \la, h_\alpha \rangle = 0$, then $L$ is a
$\ga$-module of $\mf b (\alpha)$-highest weight $\la$  and $v$ is
a $\mf b (\alpha)$-highest weight vector.

\item If $\langle \la, h_\alpha \rangle \neq 0$, then $L$ is a
$\ga$-module of $\mf b (\alpha)$-highest weight $\la -\alpha$ and
$f_\alpha v$ is a $\mf b (\alpha)$-highest weight vector.
\end{enumerate}
\end{lemma}

\begin{proof}
We first observe three simple identities:

(i) $ e_\alpha f_\alpha v = [e_\alpha, f_\alpha] v =h_\alpha v =
\langle \la, h_\alpha \rangle v. $

(ii) $ e_\beta f_{\alpha} v = [e_\beta, f_{\alpha}] v=0$ for any
$\beta \in \Phi^{+} \cap \Phi(\alpha)^{+}$, since either $\beta
-\alpha$ is not a root or it belongs to $\Phi^{+} \cap
\Phi(\alpha)^{+}$.

(iii) $f_{\alpha}^2 v =0$, since $\alpha$ is an isotropic odd root.

Now, we consider the two cases separately.

(1) Assume that $\langle \la, h_\alpha \rangle = 0$. Then we must
have $f_\alpha v =0$, otherwise $f_\alpha v$ would be a $\mf
b$-singular vector in the simple $\ga$-module $V$ by (i) and (ii).
This together with Lemma \ref{odd ref} implies that $v$ is a $\mf
b (\alpha)$-highest weight vector of weight $\la$ in the
$\ga$-module $V$.

(2) Assume that $\langle \la, h_\alpha \rangle \neq 0$. Then (i),
(ii), (iii) and Lemma \ref{odd ref} imply that $f_\alpha v$ is
nonzero and it is a $\mf b (\alpha)$-highest weight vector of
weight $\la -\alpha$ in $V$.
\end{proof}

\subsection{Finite-dimensional simple $\osp$-modules}

The case of $\ga =\spo(2m|2n+1)$ will be treated in detail (while
the case of $\ga =\spo(2m|2n)$ is similar). As usual, we have the
triangular decomposition of $\ga$ (with respect to the standard
Borel) $\ga =\n^- \oplus \h \oplus \n^+$, which allows us to
define the Verma module $\Delta (\la)$ associated to $\la
=\sum_{i=1}^{m} \la_i \delta_i + \sum_{j=1}^n \la_j \ep_j \in
\h^*$. Then $\Delta(\la)$ admits a unique simple quotient
$\ga$-module, denoted by $L(\la)$.

As for Lie algebras, a finite-dimensional simple $\ga$-module has
to be a highest weight module (with respect to any Borel), and
hence is isomorphic to some $L(\la)$, and $L(\la) \not \cong
L(\mu)$ if $\la \neq \mu$. However, the classification of
finite-dimensional simple $\ga$-modules is non-trivial, partly
because the even subalgebra of $\ga$ is not a Levi subalgebra.
Clearly a necessary condition for $L(\la)$ to be
finite-dimensional is that $\la$ is dominant integral with respect
to the even subalgebra $\ev \ga = \mf{sp}(2m) \oplus {\mf so}
(2n+1)$.

\begin{definition}\label{hook:def}
A partition $\mu =(\mu_1, \mu_2, \ldots)$, or simply a {\em hook partition} when $m,n$
are implicitly understood,  is called an {\em $(m|n)$-hook partition} if $\mu_{m+1} \le
n$.
\end{definition}
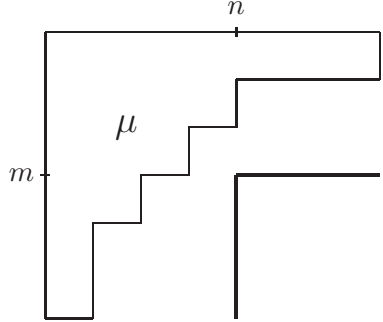
\begin{center}
\hskip 0cm \setlength{\unitlength}{0.25in}
\begin{picture}(7.5,6.5)
\put(0,0){\line(1,0){1}} \put(1,0){\line(0,1){2}} \put(1,2){\line(1,0){1}}
\put(2,2){\line(0,1){1}} \put(2,3){\line(1,0){1}} \put(3,3){\line(0,1){1}}
\put(3,4){\line(1,0){1}} \put(4,4){\line(0,1){1}} \put(4,5){\line(1,0){3}}
\put(7,5){\line(0,1){1}} \put(7,6){\line(-1,0){7}} \put(0,6){\line(0,-1){6}}
\put(1.7,4){\makebox(0,0)[c]{\Large$\mu$}} \put(-.1,3){\line(1,0){0.2}}
\put(-.5,3){\makebox(0,0)[c]{$m$}} \put(4,6.1){\line(0,-1){0.2}}
\put(4,6.5){\makebox(0,0)[c]{$n$}} \put(4,3){\linethickness{1pt}\line(0,-1){3}}
\put(4,3){\linethickness{1pt}\line(1,0){3}}
\end{picture}
\end{center}

Given an $(m|n)$-hook partition $\mu$, we denote by $\mu^+
=(\mu_{m+1}, \mu_{m+2}, \ldots)$ and write its transpose, which is
necessarily of length $\leq n$, as $\nu =(\mu^+)' =(\nu_1, \ldots,
\nu_n)$. We define the weights
\begin{eqnarray}
\mu^\natural &=& \mu_1 \delta_{1} +\ldots +\mu_m \delta_{m} +\nu_1
\ep_1 +\ldots +\nu_{n-1} \ep_{n-1} + \nu_n \ep_n\label{naturalmap}
   \\
\mu^\natural_- &=& \mu_1 \delta_{1} +\ldots +\mu_m \delta_{m} +\nu_1
\ep_1 +\ldots +\nu_{n-1} \ep_{n-1} - \nu_n \ep_n.\nonumber
\end{eqnarray}
($\mu^\natural_-$ is only used for $\spo(2m|2n)$ below).

\begin{theorem}\label{finite:hw:b}
Given any $(m|n)$-hook partition $\mu$, the simple
$\spo(2m|2n+1)$-module $L(\mu^\natural)$ (with respect to the
standard Borel) is finite-dimensional. Moreover, these modules
form a complete list of non-isomorphic finite-dimensional simple
$\spo(2m|2n+1)$-modules.
\end{theorem}

The above theorem is due to Kac \cite{K2} who formulated the
conditions for finite-dimensional highest weight simple
$\ga$-modules in terms of Dynkin labels (instead of hook
partitions). A different proof using odd reflections is given by
Shu and Wang (which also works in characteristic $p>0$). Let us
sketch the idea of a proof following \cite{SW}, as the argument
therein bears some similarity to the argument used later on for
super duality. The same type of argument works for
Theorem~\ref{finite:D} below for $\spo(2m|2n)$.

\begin{proof} [Sketch of a proof]
Let $\mu$ be an $(m|n)$-hook partition. We observe that the simple
$\spo(2M|2n+1)$-module $V$ of highest weight $\mu$ (with respect
to the standard Borel) is finite-dimensional for $M\ge \ell(\mu)$,
since it appears as a subquotient of a suitable tensor product of
simple modules of fundamental weights. Then we use odd reflections
to change the standard Borel of $\spo(2M|2n+1)$ to a Borel which
is compatible with the standard Borel of $\spo(2m|2n+1)$ (which is
regarded as a subalgebra of $\spo(2M|2n+1)$). We can show that the
highest weight of $V$ with respect to the new Borel is
$\mu^\natural$, hence by restriction to  $\spo(2m|2n+1)$ we have
proved the first part of the theorem.

A highest weight for a simple finite-dimensional $\ga$-module is
necessarily of the form $\mu_1 \delta_{1} +\ldots +\mu_m
\delta_{m} +\nu_1 \ep_1 +\ldots + \nu_n \ep_n$, where $(\mu_1,
\ldots, \mu_m)$ and $(\nu_1, \ldots, \nu_n)$ are partitions by the
dominance condition on the even subalgebra $\ev \ga$. To prove the
remaining condition $\nu'_1 \le \mu_m$, it suffices to prove it
for $\spo(2|2n+1)$ (which is a subalgebra of $\spo(2m|2n+1)$). Let
$V$ be a finite-dimensional simple $\spo(2|2n+1)$-module.
Via the sequence of odd reflections
$\delta_1-\epsilon_1,\delta_1-\epsilon_2,\ldots,\delta_1-\epsilon_n$ we change the
standard Borel to a Borel with an odd non-isotropic simple root $\delta_{1}$. The
dominance condition on the simple root $\delta_{1}$ (or rather on the even root
$2\delta_{1}$), which is imposed by the finite-dimensionality of $V$, provides the
desired necessary condition.
\end{proof}

\begin{xca}
Complete the details in the proof of Theorem \ref{finite:hw:b}.
\end{xca}

Similarly, we classify the finite-dimensional simple
$\spo(2m|2n)$-modules.

\begin{theorem}\label{finite:hw:d} \cite{K2} (cf.~\cite{SW}) \label{finite:D}
The modules $L(\mu^\natural)$ and $L(\mu^\natural_-)$(with respect
to the standard Borel), where $\mu$ runs over all $(m|n)$-hook
partitions, are all the non-isomorphic finite-dimensional simple
$\spo(2m|2n)$-modules.
\end{theorem}

Since $\ev \ga$ is not a Levi subalgebra of $\ga$, the above
definition of the Kac module for $\gl(m|n)$-module is no longer
valid for $\ga$. Nevertheless, the notion of ``typical weights" in
Definition~\ref{def:typ} still makes sense for $\osp$ (or any
basic classical Lie superalgebra). The following is the $\osp$
counterpart of Theorem~\ref{typical ch}, whose much more involved
proof will be skipped.

\begin{theorem} \cite{K2}
Let $\ga =\osp(\ell|2n)$, for $\ell =2m$ or $2m+1$. If the weight
$\la =\mu^\natural \in \h^*$ (and in addition $\la
=\mu^\natural_-$ when $\ell =2m$) for an $(m|n)$-hook partition
$\mu$ is typical, then $\ch L(\la)$ is given by the right hand
side of \eqref{typ ch}.
\end{theorem}

\begin{xca}\label{xca:osp:grading}
Recall that $\mf{so}(\ell)$ may be identified with
$\wedge^2(\C^\ell)$, and more generally, $\osp(\ell | 2n)$ may be
identified with $\wedge^2 (\C^{\ell | 2n}) = \wedge^2 (\C^{\ell})
\oplus (\C^\ell \otimes \C^{2n}) \oplus S^2 (\C^{2n})$. Now using
the invariant bilinear form to identify $\C^{\ell | 2n}$ with
$\C^{\ell |0}\oplus\C^{0|n}\oplus\C^{0|n*}$, we define a
$\Z$-gradation on $\C^{\ell | 2n}$ by setting $\text{deg}v=0$ for
$v\in\C^{\ell |0}$, $\text{deg}v=1$ for $v\in \C^{0|n}$, and
$\text{deg}v=-1$ for $v\in \C^{0|n*}$.  This induces a
$\Z$-gradation on $\wedge^2(\C^{\ell |2n})$ and hence on
$\osp(\ell |2n)$.  Prove that
\begin{align*}
\osp(\ell |2n)=\bigoplus_{i=-2}^2\osp(\ell |2n)_i,
\end{align*}
where $\osp(\ell |2n)_0\cong\mf{so}(\ell)\oplus\gl(n)$; as
$\osp(\ell | 2n)_0$-modules we have $\osp(\ell | 2n)_1\cong
\C^\ell\otimes\C^n$ and $\osp(\ell | 2n)_2\cong \C \otimes
S^2(\C^n)$.
\end{xca}


\section{Schur-Sergeev duality}
\label{sec:Schur}

\subsection{The formulation}

Let $\ga =\glmn$. Let $\{e_i| i \in I(m|n)\}$ be the standard
basis for the natural $\ga$-module $V =\C^{m|n}$.

Then $V^{\otimes d}$ is naturally a $\ga$-module by letting
\begin{align*}
\Phi_d (g).(v_1\otimes v_2 \otimes\ldots \otimes v_d) = &
g.v_1\otimes \cdots \otimes v_d + (-1)^{|g|\cdot|v_1|} v_1\otimes
g.v_2 \otimes\cdots \otimes v_d
 \\ & +\ldots
 +  (-1)^{|g|\cdot(|v_1| +\ldots +|v_{d-1}|)} v_1\otimes  v_2 \otimes\cdots
\otimes g.v_d,
\end{align*}
where $g \in \ga$ and $v_i \in V$ are assumed to be $\Z_2$-homogeneous.

On the other hand, the action of the symmetric group ${\mathfrak S}_d$ on $V^{\otimes d}$
is determined by
\begin{align*}
\Psi_d &((i,i+1)).   v_1\otimes \cdots v_i \otimes v_{i+1} \otimes
\cdots \otimes v_d \\
& = (-1)^{|v_i|\cdot |v_{i+1}|} v_1\otimes \cdots v_{i+1} \otimes
v_{i} \otimes \cdots \otimes v_d, \quad 1\le i \le d-1,
\end{align*}
where $(i,j)$ denotes a transposition in ${\mathfrak S}_d$ and $v_i, v_{i+1}$ are
$\Z_2$-homogeneous.

\begin{lemma}   \label{comm}
 The actions of $(\ga, \Phi_d)$ and $(\mf S_d, \Psi_d)$ on $V^{\otimes d}$
 commute with each other.
Symbolically, we write
$$
\ga \; \stackrel{\Phi_{d}}{\curvearrowright} \; V^{\otimes d} \;
\stackrel{\Psi_{d}}{\curvearrowleft} \; {\mathfrak S}_d.
$$
\end{lemma}

\begin{xca}
 Verify Lemma~\ref{comm}.
\end{xca}

We write $\la \vdash d$ for a partition $\la$ of $d$.
Denote by $\Pdmn$ the set of all $(m|n)$-hook partitions of size $d$ and let
\begin{equation*}
\Pmn =\cup_{d \ge 0} \Pdmn =\{\la \mid \la \text{ are partitions
with }  \la_{m+1} \le n\}.
\end{equation*}
Also set $\Pdm =\mc P(d, m|0)$. We denote by $L(\la^\natural)$ for
$\la \in \Pmn$ the simple $\ga$-module of highest weight
$\la^\natural$ with respect to the standard Borel subalgebra. For
a partition $\la \vdash d$, we denote by $S^\la$ the Specht module
of $\mathfrak S_d$. For example, $S^{(d)}$ is the trivial module,
and $S^{(1^d)}$ is the sign module $\sgn$. The following
superalgebra analog of Schur duality was due to Sergeev \cite{Sv}
(also cf. \cite{BeR}).

\begin{theorem} [Schur-Sergeev duality] \label{Sergeev}
\begin{enumerate}
\item The images $\Phi_d$ and $\Psi_d$, $\Phi_d(U(\ga))$ and
$\Psi_d(\C {\mathfrak S}_d)$, satisfy the double centralizer
property, i.e.
\begin{align*}
\Phi_{d}(U(\ga))=&\End_{{\mathfrak S}_d}(V^{\otimes d}),\\
&\End_{U(\ga)}(V^{\otimes d})
= \Psi_d(\C {\mathfrak S}_d).
\end{align*}

\item As $U(\glmn) \otimes \C \mf S_d$-module, we have
$$
V^{\otimes d} \cong \bigoplus_{\la \in \Pdmn} L(\la^\natural)
\otimes S^\la.
$$
\end{enumerate}
\end{theorem}
We refer to a $U(\gl(m|n))\otimes\C\mf{S}_d$-module as a $(\glmn, {\mathfrak
S}_d)$-module. Similar conventions also apply to similar setups below.

\begin{remark}   \label{rem:SchurD}
\begin{enumerate}

\item For $n=0$, the above Sergeev duality reduces to the usual Schur duality. If in addition
    $d=2$, $(\C^m)^{\otimes 2} =S^2 (\C^m) \oplus \wedge^2 (\C^m).$ This fits well
    with the well-known fact that as $\gl(m)$-modules $S^2 (\C^m)$ and $\wedge^2
    (\C^m)$ are, respectively, irreducible of highest weights $2\Lambda_1$ and
    $\Lambda_2$, where $\Lambda_i$ denotes the $i$th fundamental weight.

\item For $m=0$, $\la^\natural =\la'$ (the conjugate partition), and $S^\la =S^{\la'}
    \otimes \sgn$. In this case, the Sergeev duality reduces to the version of Schur
    duality twisted by the sign representation of ${\mathfrak S}_d$, i.e., as
    $(\gl(m), {\mathfrak S}_d)$-module,
$$(\C^m)^{\otimes d} \cong \bigoplus_{\mu \in \Pdm}
L(\mu) \otimes S^{\mu'}.
$$

\item If $d \le mn +m+n$, then $\Pdmn =\{\la \vdash d\}$, and
every simple ${\mathfrak S}_d$-module appears in the Sergeev
duality decomposition.

\item For $d=2$, the Sergeev duality reduces to the decomposition
    $(\C^{m|n})^{\otimes 2} =S^2 (\C^{m|n}) \oplus \wedge^2 (\C^{m|n})$, where $S^2$
    and $\wedge^2$ are understood in the super sense. In particular, as an ordinary
    vector spaces,
\begin{align*}
S^2 (\C^{m|n}) = S^2 (\C^{m}) \oplus (\C^m \otimes \C^n) \oplus
\wedge^2 (\C^{n}).
\end{align*}
\end{enumerate}
\end{remark}
\subsection{Proof of Theorem~\ref{Sergeev}}

Set $\mathcal A:=\Phi_d(U(\ga))$ and $\mathcal B :=\Psi(\C {\mathfrak S}_d)$. It is clear
that
$$
\mathcal A \subset  \End_{{\mathfrak S}_d} (V^{\otimes d})  =
(\End (V^{\otimes d}))^{{\mathfrak S}_d} \cong S^d (\End V),
$$
where $S^d(-)$ denotes the $d$th symmetric tensor. With extra work, one proves that $S^d
(\End V) \subset \mathcal A$ (the superalgebra generalization does not cause any extra
difficulty). Hence $\Phi_{d}(U(\ga))= \End_{{\mathfrak S}_d}(V^{\otimes d})$.

Since $\mathcal B =\C {\mathfrak S}_d$ is a semisimple algebra, it follows (cf.
\cite[Theorem~3.3.7]{GW}) that $\End_{U(\ga)}(V^{\otimes d}) = \Psi_d(\C {\mathfrak
S}_d).$ This proves (1).

Let $W =V^{\otimes d}$. It follows from the double centralizer property and the
semisimplicity of $\C \mf S_d$ that we have a
multiplicity-free decomposition of  the $(\glmn, \mf S_d)$-module $W$:
\begin{equation*}
W \equiv V^{\otimes d} \cong \bigoplus_{\la \in \mathfrak
P(d,m|n)} L^{[\la]} \otimes S^\la,
\end{equation*}
where $L^{[\la]}$ is some simple $\glmn$-module associated to $\la$, whose highest weight
(with respect to the standard Borel) is to be determined. Also to be determined is the
index set $\mathfrak P(d,m|n) =\{ \la \vdash d \mid L^{[\la]} \neq 0\}$.

First we need to prepare some notations.

Let $\text{Cp}(m|n)$ be the set of  pairs $\nu|\mu$ of compositions $\nu
=(\nu_{1},\ldots, \nu_{m})$ of length $\leq m$ and $\mu =(\mu_1,\ldots, \mu_n)$ of length
$\leq n$, and let
$$
\text{Cp}(d,m|n) =\{ \nu|\mu \in \text{Cp}(m|n) \mid \sum_i \nu_i
+\sum_j \mu_j =d\}.
$$
We have the following weight space decomposition (with respect to
the Cartan subalgbra of diagonal matrices $\h \subset \glmn$):
$$
W =\bigoplus_{\nu|\mu \in \text{Cp} (d,m|n)} W_{\nu|\mu},
$$
where $W_{\nu|\mu}$ has a linear basis $e_{i_1} \otimes \ldots \otimes e_{i_d}$, with the
indices satisfying the following equality of sets:
\begin{equation} \label{munu basis}
\{i_1, \ldots, i_d\} =\{\underbrace{\ov 1,\ldots, \ov
1}_{\nu_{1}}, \ldots, \underbrace{\ov m,\ldots, \ov m}_{\nu_{m}}
; \underbrace{1,\ldots, 1}_{\mu_1}, \ldots, \underbrace{n,\ldots,
n}_{\mu_n} \}.
\end{equation}
Let $\mf S_{\nu|\mu} :=\mf S_{\nu_{1}} \times \ldots \times \mf
S_{\nu_{m}} \times  \mf S_{\mu_1} \times \ldots \times \mf
S_{\mu_n}$. The span of the vector $e_{\nu|\mu}:=e_{\ov
1}^{\otimes\nu_{1}} \otimes \ldots \otimes e_{\ov
m}^{\otimes\nu_{m}} \otimes e_1^{\otimes \mu_1} \otimes \ldots
\otimes e_n^{\otimes \mu_n}$ can be identified with the
$\mf{S}_{\nu|\mu}$-module $1_\nu\otimes \text{sgn}_\mu$. Since
$\mf{S}_d e_{\nu|\mu}$ spans $W_{\nu|\mu}$ we have a surjective
$\mf{S}_d$-homomorphism from
$\text{Ind}_{\mf{S}_{\nu|\mu}}^{\mf{S}_d}(1_\nu\otimes\text{sgn}_\mu)$
onto $W_{\nu|\mu}$ by Frobenius Reciprocity. By counting the
dimensions we have an $\mf{S}_d$-isomorphism:
$$
W_{\nu|\mu} \cong \ind_{\mathfrak S_{\nu|\mu}}^{\mathfrak S_d} (1_\nu \otimes
\sgn_{\mu}). $$ Let us denote the decomposition of $W_{\nu|\mu}$ into irreducibles by
$$
W_{\nu|\mu} =\bigoplus_{\la} K_{\la, \nu|\mu} S^\la, \quad \text{ for }  K_{\la, \nu|\mu} \in \Z_+.
$$

Let $\la$ be a partition which is identified with its Young diagram.
Recall $I(m|n)$ is totally ordered by \eqref{eq:order}. A {\em hook tableau}
$T$ of shape $\la$, or an {\em
hook $\la$-tableau} $T$, is an assignment of an element in $I(m|n)$ to
each box of the Young diagram $\la$ satisfying the following conditions:
\begin{enumerate}
\item The numbers are weakly increasing along each row and column.

\item The numbers from $\{\ov 1, \ldots, \ov m\}$ are strictly increasing along each
    column.

\item The numbers from $\{1, \ldots, n\}$ are strictly increasing along each row.
\end{enumerate}
Such a $T$ is said to have {\em content $\nu|\mu \in
\text{Cp}(m|n)$} if $\bar i \in I(m|0)$ appears $\nu_i$ times and
$j \in I(0|n)$ appears $\mu_j$ times. Denote by $\mc{HT}(\la,
\nu|\mu)$ the set of hook $\la$-tableaux of
content $\nu|\mu$.

\begin{lemma}
 We have $K_{\la, \nu|\mu} = \# \mc{HT} (\la, \nu|\mu)$.
\end{lemma}

\begin{proof}
Recall that $ \ind_{\mathfrak S_{\nu|\mu}}^{\mathfrak S_d} (1_\nu \otimes
\sgn_{\mu}) \cong \bigoplus_{\la} K_{\la, \nu|\mu} S^\la$.

First assume that $\mu =\emptyset$, and we prove the formula by induction on the length
$r =\ell (\nu)$. A hook (=semistandard) tableau $T$ of shape $\la$ and content $\nu$ gives
rise to a sequence of partitions $\emptyset =\la^0 \subset \la^1 \subset \ldots \subset
\la^r =\la$ such that $\la^i$ has the shape given by the parts of $T$ with entries $\le i$,
and $\la^i /\la^{i-1}$ has $\nu_i$ boxes for each $i$. This sets up a bijection between
$\mc{HT}(\la,\nu)$ and the set of such sequences of partitions. Denote $d_1 = d-\nu_r$ and
$\wt{\nu} =(\nu_1, \ldots, \nu_{r-1})$. We have $\text{Ind}_{\mf S_{\wt{\nu}}}^{\mf
S_{d_1}}1_{\mf{S}_{\tilde{\nu}}} \cong \bigoplus_{\rho \vdash d_1} K_{\rho, \wt{\nu}}
S^\rho$, where $K_{\rho, \wt{\nu}} = \# \mc{HT}(\rho, \wt{\nu})$ by induction hyperthesis. Now
the induction step is simply the Piere's rule (c.f.~\cite[Chapter 1 (5.16)]{Mac}): for a
partition $\rho \vdash d_1$,
$$
\text{Ind}_{\mf S_{d_1} \times \mf S_{\nu_r}}^{\mf S_d}
(S^\rho \otimes 1_{\nu_r}) \cong
\bigoplus_{\la} S^\la,
$$
where $\la$ is such that $\la/\rho$ is a horizontal strip of $\nu_r$ boxes.

Then using the above special case as the initial step, we complete the proof in the general case by
induction on the length of $\mu$, in which the induction step is exactly the conjugated
Piere's rule.
\end{proof}

\begin{lemma} \label{K=0}
Let $\la \vdash d$ and $\nu|\mu \in \text{Cp}(d,m|n)$. Then $K_{\la,\nu|\mu} =0$ unless
$\la \in \Pdmn$.
\end{lemma}

\begin{proof}
By the identity $K_{\la, \nu|\mu} = \# \mc{HT} (\la, \nu|\mu)$, it suffices to prove that if a
hook $\la$-tableau $T$ of content $\nu|\mu$ exists, then $\la_{m+1}
\le n$.

By applying the hook tableau condition (2) to the first column of $T$, we see that
the first entry $k \in I(m|n)$ in row $(m+1)$ satisfies $k>0$. Applying the hook tableau
condition (3) to the $(m+1)$st row, we conclude that $\la_{m+1} \le n$.
\end{proof}

Lemma~\ref{K=0} implies that $\mf P(d,m|n) \subset \Pdmn$. On the other hand, given $\la
\in \Pdmn$, clearly a hook $\la$-tableau exists, e.g., we can fill in
the numbers $\ov 1,\ldots, \ov m$ on the first $m$ rows of $\la$ row by row downward, and
then for the (possibly) remaining rows of $\la$, we fill in the numbers $1, \ldots, n$
column by column from left to right. This distinguished $\la$-tableau will
be denoted by $T_\la$. Hence, we have proved that
$$
\mf P(d,m|n) = \Pdmn.
$$

For a given $\la \in \Pdmn$, we have $L^{[\la]} =\bigoplus_{\nu|\mu \in \text{Cp}(d,m|n)} L^{[\la]}_{\nu|\mu}$.
Among all the contents of hook $\la$-tableaux, the one for $T_\la$
corresponds to a highest weight (by the three hook tableau conditions). Hence, we
conclude that $L^{[\la]} =L(\la^\natural)$, the simple $\ga$-module of highest weight
$\la^\natural$. This completes the proof of Theorem~\ref{Sergeev}.

%
%
\subsection{}

Clearly, the character of $L(\la^\natural)$
$$
\ch L(\la^\natural) =\sum_{\nu|\mu \in \text{Cp}(m|n)} \dim
L(\la^\natural)_{\nu|\mu} \prod_{i,j} x_i^{\nu_i} y_j^{\mu_j}
$$
is a polynomial which is symmetric in $\x =(x_1,\ldots, x_m)$ and $\y =(y_1,\ldots, y_n)$,
respectively. Denote by $m_\nu(\x)$ the
monomial symmetric polynomial associated to a partition $\nu$.

The following can be read off from the above proof of
Theorem~\ref{Sergeev}.
\begin{corollary}   \label{ch=tab}
The character of $L(\la^\natural)$ is given by:
$$
\ch L(\la^\natural) =\sum_{\nu|\mu \in \text{Cp}(d,m|n)} \# \mc{HT}(\la,
\nu|\mu) \, m_{\nu} (\x) m_{\mu} (\y),
$$
where $\nu$ and $\mu$ above are partitions.
\end{corollary}

\begin{remark}
The standard basis vectors for $S^2(\C^{m|n})$ are $e_i \otimes e_j +(-1)^{|i| |j|} e_j
\otimes e_i$, where $i,j \in I(m|n)$ satisfy $i \le j <0$, $i<0<j$, or $0<i <j$ (cf.
Remark~\ref{rem:SchurD}~ (4)). They are in bijection with the hook
tableaux of shape $\la =(2)$:
\begin{center}
\hskip 0cm \setlength{\unitlength}{0.25in}
\begin{picture}(2,1)
\put(0,0){\line(1,0){2}} \put(2,0){\line(0,1){1}} \put(0,0){\line(0,1){1}}
\put(0,1){\line(1,0){2}} \put(1,0){\line(0,1){1}} \put(0.5,0.5){\makebox(0,0)[c]{$i$}}
\put(1.5,0.5){\makebox(0,0)[c]{$j$}}
\end{picture}
\end{center}
This is compatible with the identification $S^2(\C^{m|n})\cong L(\glmn,
2\delta_1)$.
\end{remark}

\begin{remark}
We sketch below a more standard argument for the standard Schur duality on $(\C^n)^{\otimes d}$,
which emphasizes the decomposition of $W$ as a $\ga$-module.

Given a composition (or a partition) $\mu$ of $d$ of length $\le n$, we denote by $W_\mu$
the $\mu$-weight space of the $\gl(n)$-module. Clearly $W_\mu$ has a basis
\begin{equation} \label{mu basis}
e_{i_1} \otimes \ldots \otimes e_{i_d}, \quad \text{ where }
\{i_1, \ldots, i_d\} =\{\underbrace{1,\ldots, 1}_{\mu_1}, \ldots,
\underbrace{n,\ldots, n}_{\mu_n} \}.
\end{equation}
As a $(\gl(n), {\mathfrak S}_d)$-module,
$$
W \cong \bigoplus_{\la} L(\la) \otimes U^{\la},
$$
where $U^\la := \Hom_{\gl(n)} (L(\la), W)\cong W_\la^{\n^+}$ (the
space of highest weight vectors in $W$ of weight $\la$). Only $\la
\in \Pdn$ can be as highest weights for $\gl(n)$ which appears in
the decomposition of $(\C^n)^{\otimes d}$, and every such $\la$
indeed appears as it appears as a summand of $\wedge^{\la'_1}
(\C^n) \otimes \wedge^{\la'_2} (\C^n)  \otimes \cdots$.

Note that $\Pdn$ has two interpretations: one as the polynomial weights for $\gl(n)$ and
the other as compositions of $d$. A remarkable fact is that the partial order on weights
induced by the positive roots of $\gl(n)$ coincides with the dominance partial order
$\unrhd$ on compositions.

Since $W =\oplus_\mu W_\mu =\oplus_{\mu,\la : \la \unrhd \mu} L(\la)_\mu \otimes U^\la$,
we conclude that, as $\mathfrak S_d$-module,
$$
W_\mu \cong \bigoplus_{\la \unrhd \mu} \dim L(\la)_\mu U^\la.
$$
On the other hand, $\mathfrak S_d$ acts on the basis of $W_\mu$ \eqref{mu basis}
transitively and the stabilizer of the basis element $e_1^{\otimes \mu_1} \otimes \ldots
\otimes e_n^{\otimes \mu_n}$ is the Young subgroup $\mf S_\mu =\mf S_{\mu_1} \times
\ldots \times \mf S_{\mu_n}$. Therefore we have
$$
W_\mu \cong \ind_{\mathfrak S_\mu}^{\mathfrak S_d} {1_\mu}
=\bigoplus_{\la \unrhd \mu} K_{\la \mu} S^\la,
$$
where $K_{\la \mu}$ is the Kostka number which satisfies $K_{\la \la} =1$.

By the double centralizer property, $U^\la$ has to be an
irreducible $\mathfrak S_d$-module for each $\la$. We compare the
above two interpretations of $W_\mu$ in the special case when
$\mu$ is dominant (i.e., a partition). One by one downward along
the dominance order, this provides the identification $U^\mu
=S^\mu$ for every $\mu$, and moreover, we obtain the well-known
equality
$$
\dim L(\la)_\mu = K_{\la \mu}.
$$
\end{remark}

\begin{definition}
Let $\mu \in \h^*$. The {\em degree of atypicality} of $\mu$ is
the maximum cardinality of a set of pairwise orthogonal  $\alpha
\in \Phi^+_{\bar 1}$ such that $(\mu+\rho,\alpha)=0$.
\end{definition}

The following observation seems to be new.
\begin{proposition}\label{atyp:prop}
Let $\la\in\mc{P}(m|n)$. The degree of atypicality of
$\la^\natural$ equals the minimal number $i$ such that $\la$
contains a partition of rectangular shape $((m-i)^{n-i})$ with
$0\le i\le \min\{m,n\}$.
\end{proposition}
As a corollary, $\la^\natural$ is typical if and only if $\la_m
\ge n$, as observed earlier in \cite{BeR}.

\begin{xca}
Prove Proposition \ref{atyp:prop}.
\end{xca}

\begin{remark} \label{qn}
There exists another interesting generalization of Schur duality
for the queer Lie superalgebras due to Sergeev \cite{Sv}.
\end{remark}

\subsection{}

Assume that the $\ep\delta$-sequence associated to a Borel
subalgebra $\mf b$ of $\glmn$ (cf. Section~\ref{sec:ep delta}) is
given by a sequence of $d_1$ $\delta$'s, $e_1$ $\ep$'s, $d_2$
$\delta$'s, $e_2$ $\ep$'s, $\ldots, d_r$ $\delta$'s, $e_r$ $\ep$'s
(possibly $d_1 =0$ or $e_r =0$). Associated to an $(m|n)$-hook
Young diagram $\la$, we define a weight $\la^{\mf b} \in \h^*$ as
follows. Take the first $d_1$ {\em row} numbers of $\la$ as the
coefficients of the first $d_1$ $\delta$'s. Denote by $\la^1$ the
Young diagram obtained from $\la$ with the first $d_1$ rows of
$\la$ removed. Take the first $e_1$ {\em column} numbers of
$\la^1$ as the coefficients of the first $e_1$ $\ep$'s. Denote by
$\la^2$ the Young diagram obtained from $\la^1$ with the first
$e_1$ columns of $\la^1$ removed. Then take the first $d_2$ {\em
row} numbers of $\la^2$ as the coefficients of the following $d_2$
$\delta$'s, and so on,  until we reach the empty partition. The
resulting weight is denoted by $\la^{\mf b}$. Below is an example
of $\la_{\mf{b}}$ for $d_1=e_1=2$ and $d_2=1$.
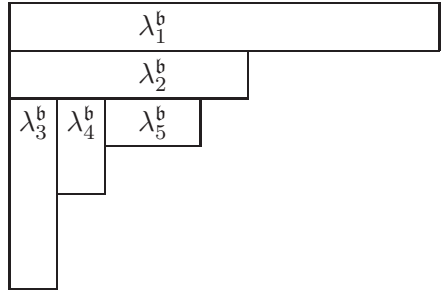
\begin{center}
\hskip 0cm \setlength{\unitlength}{0.25in}
\begin{picture}(9,6)
\put(0,0){\line(1,0){1}} \put(1,0){\line(0,1){2}} \put(1,2){\line(1,0){1}}
\put(2,2){\line(0,1){1}} \put(2,3){\line(1,0){2}} \put(4,3){\line(0,1){1}}
\put(4,4){\line(1,0){1}} \put(5,4){\line(0,1){1}} \put(5,5){\line(1,0){4}}
\put(9,5){\line(0,1){1}} \put(9,6){\line(-1,0){9}} \put(0,6){\line(0,-1){6}}
\put(0,5){\line(1,0){5}} \put(0,4){\line(1,0){4}} \put(1,4){\line(0,-1){2}}
\put(2,4){\line(0,-1){1}} \put(3,5.5){\makebox(0,0)[c]{$\lambda^{\mathfrak{b}}_1$}}
\put(3,4.5){\makebox(0,0)[c]{$\lambda^{\mathfrak{b}}_2$}}
\put(3,3.5){\makebox(0,0)[c]{$\lambda^{\mathfrak{b}}_5$}}
\put(1.5,3.5){\makebox(0,0)[c]{$\lambda^{\mathfrak{b}}_4$}}
\put(0.5,3.5){\makebox(0,0)[c]{$\lambda^{\mathfrak{b}}_3$}}
\end{picture}
\end{center}

Recall that by convention $L(\la^\natural)$ is a highest weight $\ga$-module with respect
to the standard Borel subalgebra $\mf b^{\text{st}}$.

\begin{theorem}  \cite{CLW}  \label{hwt change}
Let $\la$ be an $(m|n)$-hook partition. Let $\mf b$ be an arbitrary Borel subalgebra of $\glmn$.
Then, the $\mf b$-highest weight of the simple $\glmn$-module $L(\la^\natural)$ is
$\la^{\mf b}$.
\end{theorem}

\begin{proof}
Let us consider an odd reflection which changes a Borel subalgebra $\mf b_1$ to $\mf
b_2$. Assume the theorem holds for $\mf b_1$. We observe by Lemma~\ref{hwt odd} that the
statement of the theorem for $\mf b_2$ follows from the validity of the theorem for $\mf
b_1$. The statement of the theorem is apparently consistent with a change of Borel
subalgebras induced from a real reflection, and all Borel subalgebras of $\glmn$ are
linked by a sequence of real and odd reflections. Hence, once we know the theorem holds
for one particular Borel subalgebra, it holds for all. We finally note that the theorem
holds for the standard Borel subalgebra $\mf b^{\text{st}}$, which corresponds to the
sequence of $m$ $\delta$'s followed by $n$ $\ep$'s. It is clear that $\la^{\mf
b^{\text{st}}} =\la^\natural$.
\end{proof}

\begin{remark}
A variant of Theorem~\ref{hwt change} holds for Lie superalgebras of type $\osp$, see
\cite{CLW}.
\end{remark}

\begin{example}  \label{ex:3hw}
Let us describe the highest weights in Theorem~\ref{hwt change} with respect to the three
Borel subalgebras of special interest.
\begin{enumerate}
\item As seen above, $\la^{\mf b^{\text{st}}} =\la^\natural$.

\item If we take the opposite Borel subalgebra $\mf b^{\text{op}}$ corresponding to a
    sequence of $n$ $\ep$'s followed by $m$ $\delta$'s, then $\la^{\mf b^{\text{op}}}
    =(\la')^\natural$, where $\natural$ applies to a $(n|m)$-hook partition (instead
    of $(m|n)$-hook partition). This is as expected from Schur-Sergeev duality
    Theorem~\ref{Sergeev}.

\item In case when $|m -n| \le 1$, we may take a Borel subalgebra $\mf b^{\text{o}}$
    whose simple roots are all odd (or equivalently, the corresponding $\ep\delta$- sequence is
alternating between $\ep$ and $\delta$): 
\begin{center}
\hskip -3cm \setlength{\unitlength}{0.16in}
\begin{picture}(24,3)
\put(8,2){\makebox(0,0)[c]{$\bigotimes$}}
\put(10.4,2){\makebox(0,0)[c]{$\bigotimes$}}
\put(14.85,2){\makebox(0,0)[c]{$\bigotimes$}}
\put(17.25,2){\makebox(0,0)[c]{$\bigotimes$}}
\put(19.4,2){\makebox(0,0)[c]{$\bigotimes$}} \put(8.4,2){\line(1,0){1.55}}
\put(10.82,2){\line(1,0){0.8}} \put(13.2,2){\line(1,0){1.2}}
\put(15.28,2){\line(1,0){1.45}} \put(17.7,2){\line(1,0){1.25}}
\put(12.5,1.95){\makebox(0,0)[c]{$\cdots$}}
\end{picture}
\end{center}
In this case Theorem~\ref{hwt change} reduces to \cite[Theorem 7.1]{CW3}, which
states the coefficients of $\delta$ and $\ep$ in $\la^{\mf b^{\text{o}}}$ are given
by the {\em modified Frobenius coordinates} $(p_i | q_i)_{i \ge 1}$ of the partition
$\la$ (respectively, $\la'$), when the first simple root is of the form $\delta -\ep$
(respectively, $\ep -\delta$). Here by modified Frobenius coordinates we mean
$$
p_i =\max \{\la_i -i+1,  0\}, \quad q_i =\max \{\la_i'-i, 0\}
$$
so that $\sum_i (p_i + q_i) =|\la|$; ``modified" here refers to a shift by $1$ from
the $p_i$ coordinates defined in \cite[Chapter~ 1, Page~ 3]{Mac}.

For $\la =(7,5,4,3,1)$, we have $(p_1,p_2,p_3 | q_1, q_2, q_3) =(7,4,2 |4,2,1)$.
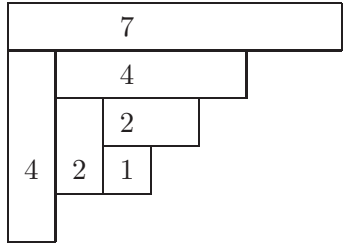
\begin{center}
\hskip 0cm \setlength{\unitlength}{0.25in}
\begin{picture}(7,5.5)
\put(0,0){\line(1,0){1}} \put(1,0){\line(0,1){1}} \put(1,1){\line(1,0){2}}
\put(3,1){\line(0,1){1}} \put(4,2){\line(0,1){1}} \put(3,3){\line(1,0){2}}
\put(5,3){\line(0,1){1}} \put(5,4){\line(1,0){2}} \put(7,4){\line(0,1){1}}
\put(7,5){\line(-1,0){7}} \put(0,5){\line(0,-1){5}} \put(0,4){\line(1,0){5}}
\put(1,4){\line(0,-1){3}} \put(1,3){\line(1,0){2}} \put(2,3){\line(0,-1){2}}
\put(2,2){\line(1,0){2}} \put(2.5,4.5){\makebox(0,0)[c]{$7$}}
\put(2.5,3.5){\makebox(0,0)[c]{$4$}} \put(2.5,2.5){\makebox(0,0)[c]{$2$}}
\put(0.5,1.5){\makebox(0,0)[c]{$4$}} \put(1.5,1.5){\makebox(0,0)[c]{$2$}}
\put(2.5,1.5){\makebox(0,0)[c]{$1$}}
\end{picture}
\end{center}

\end{enumerate}
\end{example}

\begin{xca}
Let $\la =(7,2,2,1,1)$ be an $(1|2)$-hook partition (i.e.~$\la_2 \le 2$). Write down all
the extremal weights with respect to the six Borel subalgebras of $\gl(1|2)$ from
Example~\ref{gl12} for $L(\gl(1|2), \la^\natural)$.
\end{xca}

\section{Howe duality for Lie superalgebras of type $\gl$}
\label{sec:Howe}

\subsection{The $(\glmn, \gl(d))$-Howe duality}

Let $V =\ev V \oplus \od V$ be a vector superspace. The symmetric algebra $S(V)$ is
understood as $S(V) = S(\ev V) \otimes \wedge (\od V)$. It follows that $S^k(V) =
\bigoplus_{\ell =0}^k S^\ell (\ev V) \otimes \wedge^{k -\ell} (\od V)$.

The commuting action of $\glmn$ and $\gl (d)$ on the super space $\C^d \otimes\C^{m|n}$
induces a corresponding commuting action on its symmetric algebra $S(\C^d
\otimes\C^{m|n})$.

\begin{theorem} \cite{H1, CW1, Sv2}   \label{supergl-duality}
\begin{enumerate}
\item The images of $U(\gl (d))$ and $U(\glmn)$ in $\End \big(S (\C^d
    \otimes\C^{m|n}) \big)$ satisfy the double centralizer property.

\item As a $(\gl (d), \glmn)$-module,
$$
S (\C^d \otimes\C^{m|n}) = \sum_{\la \in \Pmn, \ell (\la) \le d}
L_d (\la) \otimes L_{m|n} (\la^\natural ).
$$
Here and below we use subscripts $m|n$ and $d$ to indicate the  (super)algebras
under consideration.
\end{enumerate}
\end{theorem}
The part~(2) of the theorem is equivalent to
$$
S^k (\C^d \otimes\C^{m|n}) = \sum_{\stackrel{\la \vdash k}{\la_{m+1} \le n,
\ell (\la) \le d}} L_d (\la) \otimes L_{m|n} (\la^\natural ), \quad k \in \N.
$$

\begin{proof}
Let us prove the equivalent version for $ S^k (\C^{m|n} \otimes \C^d)$ for (2). The
argument below is well-known (see \cite{H2}).

By the definition of the $k$-th supersymmetric algebra we have
\begin{equation*}
S^k (\C^d \otimes\C^{m|n}) \cong \left( (\C^{d})^{\otimes^k}
\otimes (\C^{m|n})^{\otimes^k} \right)^{\Delta_k},
\end{equation*}
where $\Delta_k$ is  the diagonal subgroup of $\mf S_k\times \mf S_k$ and
$(-)^{\Delta_k}$ denotes the $\Delta_k$-invariant subspace. By Theorem~\ref{Sergeev} we
have therefore

\begin{align*}
 S^k (\C^d \otimes\C^{m|n}) &
 \cong
 \Big((\sum_{|\mu|=k} L_{d}(\mu) \otimes S^{\mu})
 \otimes
 (\sum_{|\lambda|=k}L_{m|n} (\lambda^\natural) \otimes
 S^{\lambda}) \Big)^{\Delta_k}\\
 &\cong
 \sum_{|\lambda|=|\mu|=k}
 \left ( L_{d}(\mu) \otimes L_{m|n} (\lambda^\natural)
 \right)\otimes (S^{\mu}\otimes S^{\la})^{\Delta_k}\\
 &\cong
 \sum_{|\lambda|=k} L_{d}(\la) \otimes L_{m|n} (\lambda^\natural).
\end{align*}
The last equality follows since the Specht module $S^{\lambda}$
is self-contragredient.

Part (1) follows from the decomposition in (2).
\end{proof}

\begin{remark}\label{classical symm Howe}
In case $n=0$, we have $\la^\natural =\la$, and Theorem~\ref{supergl-duality} reduces to
the classical $(\gl (d), \gl(m))$-Howe duality on a symmetric tensor space (see
\cite[Section 2.1.2]{H2}).

For $m=0$, we have $\la^\natural =\la'$, and
Theorem~\ref{supergl-duality} reduces to the classical $(\gl (d),
\gl(n))$-Howe duality on an exterior space $\wedge (\C^d \otimes
\C^n)$ (see \cite[Section 4.1.1]{H2}).
\end{remark}

\begin{remark}
Similarly, using Sergeev's duality for queer Lie superalgebra $\mf q (n)$ (see
Remark~\ref{qn}), we can establish a $(\mf q(m), \mf q(n))$-Howe duality, see \cite{CW2, Naz,
Sv2}.
\end{remark}

\subsection{Hook Schur polynomials as $\glmn$-characters}

Let $\x =(x_1,\ldots,x_m)$ and $\y =(y_1,\ldots,y_n)$ be formal variables. For two
partitions $\la$ and $\mu$ with $\mu\subseteq\la$ denote by $s_{\la/\mu}(\x)$ the skew
Schur polynomial associated to the skew Young diagram $\la/\mu$ such that $\ell (\la) \le
m$. The skew Schur polynomials specialize to the (usual) Schur polynomials for $\mu
=\emptyset$: $s_{\la/\emptyset}(\x) =s_{\la}(\x)$. We refer to Macdonald \cite{Mac} for
more on symmetric functions.

The {\em hook Schur polynomials} $hs_\la (\x,\y)$ in variables $\x$ and $\y$ are defined as
\begin{align}\label{mn:hook}
hs_\la (\x,\y) =
\sum_{\mu\subseteq\la} s_\mu (\x) s_{\la'/\mu'} (\y), \qquad \la \in \Pmn.
\end{align}
This is one of the several equivalent definitions, see \cite{BeR}. We should compare this
definition with the classical identity for (skew) Schur functions
\begin{equation} \label{skewSchur}
s_\la (\x,\y) =
\sum_{\mu\subseteq\la} s_\mu (\x) s_{\la/\mu} (\y).
\end{equation}

The character of the $\glmn$-module $L_{m|n} (\la^\natural)$ is by definition the trace
of the action of the diagonal matrix ${\rm diag}(x_1,\ldots,x_m;y_1,\ldots,y_n)$  on
$L_{m|n} (\la^\natural)$. The following theorem was obtained in \cite{BeR}, where the
notion of hook Schur polynomials was formulated, and it is actually equivalent to the
combinatorial formula given in Corollary~\ref{ch=tab}. We offer a  different proof below
using Howe duality (cf.~\cite{CL1}).
\begin{theorem}
Let $\la \in \Pmn.$ The character  of $L_{m|n} (\la^\natural)$ is given by
$$
\ch L_{m|n} (\la^\natural) = hs_\la (\x,\y).
$$
\end{theorem}

\begin{proof}
Let $\underline{u} =(u_1,\ldots,u_d)$ be another set of formal variables. It is well
known that $s_\la(\underline{u})$ is the character of the $\gl(d)$-module $L_d(\la)$,
i.e.~the trace of the diagonal matrix ${\rm diag}(u_1,\ldots,u_d)$. Thus, comparing the
characters of both sides of Theorem~\ref{supergl-duality}~(2), we obtain the following
combinatorial identity:
\begin{equation}  \label{identity A}
\sum_{\la \in \Pmn, \ell (\la) \le d} s_{\la}(\underline{u}) \,
\ch L_{m|n} (\la^\natural)= \prod_{i \le m,j \le n,k \le
d}(1-x_iu_{k})^{-1} (1+y_ju_k).
\end{equation}

Let $X=(x_1, x_2, \ldots)$ and $Y =(y_1, y_2, \ldots)$ denote two sets of infinitely many
variables. Recall the classical Cauchy identity
\begin{equation}  \label{Cauchy}
\sum_{\ell (\la) \le d}s_\la(\underline{u})s_\la(X)=\prod_{k \le
d}\prod_{i \ge 1}(1-x_iu_k)^{-1},
\end{equation}
which readily follows from Howe duality Theorem~\ref{supergl-duality} (by setting $n=0$
and letting $m \to \infty$).

Consider the involution $\omega$ on the ring of symmetric functions in the set of
variables $(x_{m+1}, x_{m+2}, \ldots)$. Applying $\omega$ to \eqref{Cauchy},
replacing  $(x_{m+1}, x_{m+2}, \ldots)$ by $Y$, and using
\eqref{skewSchur}, we obtain
\begin{equation*}
\sum_{\ell (\la) \le d} s_{\la}(\underline{u}) hs_\la (\x,Y) =
\prod_{i \le m,k \le d, j \ge 1}(1-x_iu_{k})^{-1}
(1+y_ju_k).
\end{equation*}
Now setting $y_{n+1} =y_{n+2} =\ldots =0$ in the above identity, we obtain that
\begin{equation*}
\sum_{\ell (\la) \le d} s_{\la}(\underline{u})  hs_\la (\x,\y) =
\prod_{i \le m,j \le n,k \le d}(1-x_iu_{k})^{-1}
(1+y_ju_k).
\end{equation*}
By comparing with \eqref{identity A} and noting the linear independence of
$\{s_{\la}(\underline{u})\}$ for $d=\infty$, we have proved the theorem.
\end{proof}

\begin{remark}
It follows that
$$hs_\la (X, \emptyset) =s_{\la}(X), \quad hs_\la
(\emptyset, Y) =s_{\la'}(Y), \quad hs_\la (Y,X)=hs_{\la'} (X, Y).
$$
Both the usual Cauchy identity and its dual version are obtained from specializations of
the corresponding character identity of Theorem~\ref{supergl-duality}.
\end{remark}

\subsection{Formulas for highest weight vectors}

We let $e^1,\ldots,e^d$ denote the standard basis for the natural $\gl(d)$-module $\C^d$,
and recall that $e_i, i \in I(m|n),$ denote the standard basis for the natural
$\gl(m|n)$-module $\C^{m|n}$. We set
\begin{equation}\label{generators}
x_a^i:=e^i \otimes e_{\bar a} \ (1\le a\le m);\quad
\eta_b^i:=e^i\otimes e_b \ (1\le b \le n).
\end{equation}

We will denote by $\Cx$ the polynomial superalgebra generated by \eqref{generators}. The
commuting actions on $\Cx$ of $\gl(d)$ and $\gl(m|n)$ may be realized in terms of first
order differential operators \eqref{gld} and \eqref{glmn} respectively ($1\le i,i'\le d$
and $1\le s,s'\le m; 1\le k,k'\le n$):
\begin{align}
E^{ii'}:=&\sum_{j=1}^m x_{j}^{i}\frac{\partial}{\partial
x_j^{i'}}+\sum_{j=1}^n\eta_j^{i}\frac{\partial}{\partial\eta_j^{i'}},
 \label{gld}\\
 \sum_{j=1}^d x_{s}^{j}\frac{\partial}{\partial x_{s'}^j},\quad
&\sum_{j=1}^d\eta_{k}^{j}\frac{\partial}{\partial\eta_{k'}^j},\quad
\sum_{j=1}^d x_s^j\frac{\partial}{\partial\eta_k^j},\quad
\sum_{j=1}^d\eta_k^j\frac{\partial}{\partial x_s^j}.\label{glmn}
\end{align}
In particular, as $(\gl (d), \glmn)$-module, $S(\C^p\otimes\C^{m|n})$ is isomorphic to
$\Cx$. The root vectors corresponding to the simple roots of $\gl (d)$ and $\glmn$ are
\begin{align}
 &\sum_{j=1}^m x_{j}^{i-1} \frac{\partial}{\partial x_j^{i}}
+\sum_{j=1}^n\eta_j^{i-1}\frac{\partial}{\partial\eta_j^{i}},
\label{radicalp}\\
  \sum_{j=1}^d
x_{s-1}^{j} & \frac{\partial}{\partial x_{s}^j},\quad
\sum_{j=1}^d \eta_{k-1}^{j}\frac{\partial}{\partial\eta_k^j},\quad
\sum_{j=1}^d
x_m^j\frac{\partial}{\partial\eta_1^j}\label{radicalpmn},
\end{align}
respectively.

Let $\la$ be an $(m|n)$-hook partition of size $k$ and of length at most $d$. We are
looking for the joint highest weight vector (with respect to the standard Borel
subalgebras), or equivalently the vector annihilated by \eqref{radicalp} and
\eqref{radicalpmn}, for the highest weight module $L_d (\la) \otimes L_{m|n}
(\la^\natural)$ of $\gl(d)\times \glmn$ appearing in the decomposition of $\Cx$. Such a
vector is unique up to a scalar multiple, thanks to the multiplicity-free decomposition
in Part 2 of Theorem~\ref{supergl-duality}.

For $1\le r\le m$, define
\begin{equation}\label{deltar}
\diamondsuit_r:={\rm det}\begin{pmatrix}
x_{1}^{1}&x_{2}^{1}&\cdots&x_{r}^{1}\\
x_{1}^{2}&x_{2}^{2}&\cdots&x_{r}^{2}\\
\vdots&\vdots&\vdots&\vdots\\
x_{1}^{r}&x_{2}^{r}&\cdots&x_{r}^{r}\\
\end{pmatrix}.
\end{equation}

Let us assume for now that $d>m$ (the case of $d \leq m$ is trivially included with much
simplification, in which case we will not need the definition of $\diamondsuit_{k,r}$
below). Under this assumption, the condition $\la_{m+1}\le n$ is no longer an empty
condition.
Recall $\la' =(\lambda_1',\lambda_2',\ldots)$ denotes the transposed partition of $\la$.
We have $d \ge \la'_1 \ge \la'_2 \ge \ldots $ and $m\ge\la'_{n+1}$.

Recall that the {\em row determinant} of an $r \times r$ matrix with possibly
non-commuting entries $A = [a_i^j]$ is defined to be
$$
\text{rdet}A = \sum_{\sigma\in \mf
S_r}(-1)^{\ell(\sigma)}a_1^{\sigma(1) }a_2^{\sigma (2) }\cdots
a_r^{\sigma (r) }.
$$
For $m\le r\le d$, we introduce the following row determinant:
\begin{equation} \label{eq_det}
\diamondsuit_{k,r}:={\rm rdet}
\begin{pmatrix}
x_1^1&x_1^2&\cdots &x_1^r\\
x_2^1&x_2^2&\cdots &x_2^r\\
\vdots&\vdots&\cdots &\vdots\\
x_m^1&x_m^2&\cdots &x_m^r\\
\eta_k^1&\eta_k^2&\cdots &\eta_k^r\\
\vdots&\vdots&\cdots &\vdots\\
\eta_k^1&\eta_k^2&\cdots &\eta_k^r\\
\end{pmatrix},\quad k=1,\ldots,n,
\end{equation}
where
the last $(r-m)$ rows are filled with the same vector
$(\eta_k^1,\eta_k^2,\ldots,\eta_k^r)$.

\begin{remark}
The determinant (\ref{eq_det}) is always nonzero. It reduces to (\ref{deltar}) when $m
=r$, and, up to a scalar multiple, it reduces to $\eta_k^1 \cdots \eta_k^r$ when $m=0$.
\end{remark}

Now let $r$ be specified by the conditions $\la'_r>m$ and
$\la'_{r+1}\le m$. Denote by $\la_{\le r}$ the subdiagram of the
Young diagram $\la$ which consists of the first $r$ columns of
$\la$, i.e., the columns of length  $>m$.
A formula for the highest weight vector associated to the Young diagram $\la_{\le r}$ is
given by the following lemma.

\begin{lemma}
The vector $\prod_{k=1}^r\diamondsuit_{k,\la'_k}
$ is a highest weight vector  for the highest weight module $L_d (\la_{\le r}) \otimes
L_{m|n} (\la_{\le r}^\natural)$ in the decomposition of $\Cx$.
\end{lemma}
\begin{proof}
Set $\diamondsuit =\prod_{k=1}^r\diamondsuit_{k,\la'_k}$, and $\mu =\la_{\le r}$. We
verify the following.

(i)  $ \diamondsuit$ has weight $(\mu, \mu^\natural)$ with respect to the action of $\gl(d)\times \gl(m|n)$.

(ii) $\diamondsuit$ is non-zero.

(iii) $\diamondsuit$ is annihilated by the operators in \eqref{radicalp}.

By Exercise~\ref{xca:hwv} below, (i)-(iii) imply that $\diamondsuit$ is also a highest
weight vector for $\glmn$.
\end{proof}

\begin{xca}  \label{xca:hwv}
Assume that $v \in \Cx$ has weights $\la$ and $\la^\natural$ with respect to $\gl(d)$ and
$\glmn$, respectively. Prove that if $v$ is a highest weight vector for $\gl (d)$, then
so is it  for $\glmn$.
\end{xca}

\begin{theorem} \cite{CW1}
\label{gldmn-hwv}
Let $\la$ be a Young diagram of length  $\leq d$ such that $\la_{m+1}\le n$.  Then, a
highest weight vector of weight $(\la, \la^\natural)$ in the $\gl(d) \times \glmn$-module
$\Cx$ is given by
$$
\prod_{k=1}^r\diamondsuit_{k,\la'_k}\prod_{j=r+1}^{\la_1}\diamondsuit_{\la'_j},
$$
where $r$ is defined by $\la'_r>m$ and $\la'_{r+1}\le m$.
\end{theorem}

\begin{proof}
Since the expression is a product of two highest weight vectors, it is a highest weight
vector. Also it is easy to verify that it has the correct weight.
\end{proof}

\section{Howe duality for Lie superalgebras of type $\osp$}
\label{sec:Howe app}

\subsection{General ideas of dual pairs}
\label{general dual}

Let $V$ be a vector (super)space, and $\mc P(V)$ the polynomial (super)algebra on $V$.
Let $\mc D (V)$ denote the Weyl (super)algebra of differential operators with polynomial
coefficients on $V$. For example, if $x_1, \ldots, x_m, \xi_1, \ldots, \xi_n$ are the
even/odd coordinates on $V$ relative to a homogeneous basis (so that $V \cong \C^{m|n}$),
then $\mc D (V)$ is generated by the even generators $x_i, \partial_i :=\partial /
\partial x_i$ and the odd generators $\xi_j, \eth_j:=
\partial /\partial \xi_j$ $(1\le i \le m, 1\le j \le n)$,
subject to the only nontrivial (super) commutation relations among the generators:
$$
\partial_i x_i -x_i
\partial_i = 1, \qquad
\eth_j \xi_j + \xi_j \eth_j = 1.
$$

Given a reductive group $G$ acting on $V$, we have a multiplicity-free decomposition over
$(G, \mc D(V)^G)$:
$$
\mc P(V) =\bigoplus_{\la \in \mf P} L(\la) \otimes U^\la,
$$
where $L(\la)$ are pairwise non-isomorphic irreducible $G$-modules and $U^\la$ are
pairwise non-isomorphic irreducible $\mc D(V)^G$-modules, as $\la$ runs over some set
$\mf P$ of parameters.

If there exists a generating set  $\{T_1, \ldots, T_r\}$ for the (super)algebra $\mc
D(V)^G$, such that $\ga ' := \C$-{span}$\langle T_1, \ldots, T_r \rangle$ forms a Lie
subalgebra of $\mc D(V)^G$, then the above $U^\la$'s are pairwise non-isomorphic
irreducible $\ga '$-modules. We refer to \cite{GW} for more detailed discussion.

\begin{remark} \label{choice}
For a suitable choice of $V$ and $G$, there could exist a very canonical candidate for
the dual partner $\ga'$, which is suggested by the First Fundamental Theorem of invariant
theory \cite{H1, GW}. We will refer to $(G, \ga')$ as a {\em Howe dual pair} acting on
$S(V)$.

Once $G, \ga'$ are fixed, the basic problems in $(G, \ga')$-Howe duality include
describing the $\ga'$-module $U^\la$ explicitly (e.g.~describing its highest weight if this makes
sense) and determining the parameter set $\mf P$.
\end{remark}

For example, if we let $V = (\C^d \otimes \C^n)^*$ and $G =\text{GL}(d)$, then $\mc P(V)
=S(\C^d \otimes \C^n)$ and we can choose $\ga' =\gl (n)$. This gives rise to the
$(\gl(d), \gl(n))$-duality from Remark \ref{classical symm Howe}.

\subsection{}
In the type $A$ Howe dualities, we can replace the Lie group $\text{GL}(d)$ by its Lie
algebra $\gl(d)$ and vice versa without loss of information. However for more general
Howe dualities, it is essential to use Lie groups $G$ as formulated in Section~\ref{general dual}.
Associated to a given Lie algebra $\ga$ (e.g. $\mf{so}(d)$), one associates several
different (possibly disconnected) groups $G$, e.g. $\text{SO}(d), \text{O}(d),
\text{Spin}(d), \text{Pin}(d)$. The Howe duality formulation favors $\text{O}(d)$ and
$\text{Pin}(d)$ in the sense of Remark~\ref{choice}.

To avoid the complication of describing the irreducible representations of different
covering groups, we will choose to discuss a Howe duality involving the Lie group
$\Sp(d)$, where $d =2\ell$ must be an even integer. In this case, the finite-dimensional
simple modules over the group $\Sp(d)$ and those over the Lie algebra $\mf{sp}(d)$
correspond bijectively, and it makes no difference if we use $\mf{sp}(d)$ to replace
$\Sp(d)$.  Besides, they admit a nice parametrization in terms of partitions. We denote
the $\text{Sp}(d)$-module corresponding to the partition $\la$ by $L(\text{Sp}(d),\la)$.

According to \cite{H1}, there exists a Howe dual pair $(\Sp(d), \mf{so}(2m))$ on $S(\C^d
\otimes \C^m)$, and a Howe dual pair $(\Sp(d), \mf{sp}(2n))$ on $\wedge(\C^d \otimes
\C^n)$. By fusing these two dual pairs together, we obtain a Howe dual pair $(\Sp(d),
\osp(2m|2n))$ on $S(\C^d \otimes \C^{m|n})$. Let us explain the commuting actions below.

As before, we identify $S (\C^d \otimes\C^{m|n}) = \Cx$. The action of the Lie algebra
$\mf{sp} (d)$ as the subalgebra of $\gl (d)$ on $\Cx$ lifts to an action of Lie group
$\Sp (d)$. On the other hand, the following action of $\glmn$ on $\Cx$ is obtained from
\eqref{glmn} by a shift of scalars on the diagonal matrices:
\begin{align}
E^{xx}_{is}=\sum_{j=1}^d x^j_i\frac{\partial}{\partial
x_s^j}+\frac{d}{2}\delta_{is},\quad
E^{x\eta}_{ik}=\sum_{j=1}^d x^j_i\frac{\partial}{\partial \eta_k^j},
\label{glmn'}\allowdisplaybreaks\\
E^{\eta x}_{ki}=\sum_{j=1}^d \eta^j_k\frac{\partial}{\partial
x_i^j},\quad E^{\eta\eta}_{tk}=\sum_{j=1}^d
\eta^j_t\frac{\partial}{\partial
\eta_k^j}-\frac{d}{2}\delta_{ik},\nonumber
\end{align}
where $i,s=1,\ldots,m$ and $k,t=1,\ldots,n$. Introduce the following additional operators
\begin{align}
\begin{split} \label{extra gen}
& I^{xx}_{is}=\sum_{j=1}^{\frac{d}{2}} \Big{(}x^j_i
x^{d+1-j}_s-x^{d+1-j}_i x^j_s\Big{)},\quad
 I^{x\eta}_{ik}=\sum_{j=1}^{\frac{d}{2}} \Big{(}x^j_i
\eta^j_k-x^{d+1-j}_i
\eta^{d+1-j}_k\Big{)},\allowdisplaybreaks\\
& I^{\eta\eta}_{kt}=\sum_{j=1}^{\frac{d}{2}} \Big{(}\eta^j_k
\eta^{d+1-j}_t-\eta^{d+1-j}_k \eta^j_t\Big{)}, \
 \Delta^{xx}_{is}=\sum_{j=1}^{\frac{d}{2}}
\Big{(}\frac{\partial}{\partial x^j_i} \frac{\partial}{\partial
x^{d+1-j}_s}-\frac{\partial}{\partial x^{d+1-j}_i}
\frac{\partial}{\partial
x^j_s}\Big{)},\allowdisplaybreaks\\
& \Delta^{x\eta}_{ik}=\sum_{j=1}^{\frac{d}{2}}
\Big{(}\frac{\partial}{\partial x^j_i} \frac{\partial}{\partial
\eta^{d+1-j}_k}-\frac{\partial}{\partial x^{d+1-j}_i}
\frac{\partial}{\partial \eta^j_k}\Big{)},\allowdisplaybreaks\\
& \Delta^{\eta\eta}_{kt}=\sum_{j=1}^{\frac{d}{2}}
\Big{(}\frac{\partial}{\partial \eta^j_k} \frac{\partial}{\partial
\eta^{d+1-j}_t}-\frac{\partial}{\partial \eta^{d+1-j}_k}
\frac{\partial}{\partial \eta^j_t}\Big{)},
\end{split}
\end{align}
where $1\le i< s\le m$ and $1\le k\le t\le n$. It is not hard to see that these operators
together with \eqref{glmn'} form a basis for the Lie superalgebra $\osp(2m|2n)$. Moreover
the actions of $\osp(2m|2n)$ and  of $\Sp(d)$ on $\Cx$ commute.

\begin{theorem} \cite{H1}
 \label{Howe:sp-osp}
Let $d =2 \ell$. The images of the algebras $\C \Sp (d)$ and
$U(\osp(2m|2n))$ in $\End \big(S (\C^d \otimes\C^{m|n}) \big)$
satisfy the double centralizer property.
\end{theorem}
\begin{proof}
The proof is based on the fact that the invariants of the classical group $\Sp(d)$ of the
corresponding dual pair in the endomorphism ring of $S(\C^{d}\otimes\C^{m|n})$ are
generated by quadratic invariants.
\end{proof}

Denote by $\mf u^+$ (respectively, $\mf u^-$) the subalgebra of $\osp(2m|2n)$ spanned by
the $\Delta$ (respectively, $I$) operators in \eqref{extra gen}. The highest weight
module $L (\osp(2m|2n), \mu)$ below is understood to be relative to the Borel subalgebra
of $\osp(2m|2n)$ corresponding to the simple roots listed in the following Dynkin
diagram:

\begin{center}
\hskip -3cm \setlength{\unitlength}{0.16in}
\begin{picture}(24,4)
\put(6,3.8){\makebox(0,0)[c]{$\bigcirc$}} \put(6,.3){\makebox(0,0)[c]{$\bigcirc$}}
\put(8,2){\makebox(0,0)[c]{$\bigcirc$}} \put(10.4,2){\makebox(0,0)[c]{$\cdots$}}
\put(12.5,1.95){\makebox(0,0)[c]{$\bigcirc$}}
\put(14.85,2){\makebox(0,0)[c]{$\bigotimes$}} \put(17.25,2){\makebox(0,0)[c]{$\bigcirc$}}
\put(19.4,2){\makebox(0,0)[c]{$\cdots$}} \put(21.4,1.95){\makebox(0,0)[c]{$\bigcirc$}}
\put(8.4,2){\line(1,0){1.1}} \put(11.02,2){\line(1,0){0.9}}
\put(13.2,2){\line(1,0){1.2}} \put(15.28,2){\line(1,0){1.45}}
\put(17.7,2){\line(1,0){0.9}} \put(19.9,2){\line(1,0){1.0}}
\put(7.6,2.2){\line(-1,1){1.3}} \put(7.6,1.8){\line(-1,-1){1.3}}
\put(9.4,1){\makebox(0,0)[c]{\tiny$\delta_{2}-\delta_{3}$}}
\put(14.8,1){\makebox(0,0)[c]{\tiny $\delta_{m}-\epsilon_{1}$}}
\put(21.4,1){\makebox(0,0)[c]{\tiny $\epsilon_{n-1} -\epsilon_{n}$}}
\put(3.2,3.8){\makebox(0,0)[c]{\tiny$\delta_{1}-\delta_{2}$}}
\put(3.2,0.3){\makebox(0,0)[c]{\tiny $-\delta_{1}-\delta_{2}$}}
\end{picture}
\end{center}
Then $\glmn$ is a Levi subalgebra of $\osp(2m|2n)$ corresponding to the removal of the
simple root $-\delta_{1}-\delta_{2}$. We have a triangular decomposition of Lie
superalgebra $\osp(2m|2n) =\mf{u}^- \oplus \glmn \oplus \mf{u}^+$.

Recall that for $\la \in \Pmn$,
$\la^\natural$ is a weight for $\glmn$. Then $\la^\natural + \ell \bf{1}$ can be regarded
as a weight for Lie superalgebras $\glmn$ and $\osp(2m|2n)$ (which share the same
Cartan subalgebra), where
$$
{\bf 1}=\sum_{i=1}^m\delta_i-\sum_{j=1}^n\ep_j.
$$
The simplified proof below of the following theorem of \cite{CZ}
is borrowed from \cite{CKW}.

\begin{theorem} \label{Sp-osp-Howe}
Let $d =2 \ell$. As an $(\Sp(d), \osp(2m|2n))$-module,
$$
S (\C^d \otimes\C^{m|n}) = \bigoplus_{\la\in\mc{P}(m|n),\ell(\la)\le\ell} L({\Sp(d)},
\la) \otimes L (\osp(2m|2n), \la^\natural + \ell \bf{1}).
$$
\end{theorem}
\begin{proof}
An element $f\in\Cx$ is called {\em harmonic}, if $f$ is
annihilated by the subalgebra $\mf{u}^+$. The space of harmonics
will be denoted by ${}^{\Sp}H$ and it evidently admits an action
of $\Sp(d)\times \glmn$. Furthermore, since
$S(\C^{d}\otimes\C^{m|n})$ is a completely reducible
$\glmn$-module, $L(\osp(2m|2n),\mu)^{\mf{u}^+}$ is also completely
reducible over $\glmn$, for any irreducible $\osp(2m|2n)$-module
$L(\osp(2m|2n),\mu)$ that appears in $S(\C^{d}\otimes\C^{m|n})$.
By irreducibility of $L(\osp(2m|2n),\mu)$ we must have
\begin{align*}
L(\osp(2m|2n),\mu)^{\mf{u}^+}\cong L_{m|n}(\mu).
\end{align*}
So, by Theorem~\ref{Howe:sp-osp} $(\text{Sp}(d),\glmn)$ forms a dual pair on the space of
harmonics ${}^{\Sp}H$. Thus, proving the theorem is equivalent to establishing the
following decomposition of ${}^{\Sp}H$ as an $\Sp(d)\times \glmn$-module:
\begin{equation}  \label{harmonic}
{}^{\Sp}H\cong\bigoplus_{\la \in \Pmn, \ell(\la) \le \ell}
L({\Sp(d)}, \la) \otimes L_{m|n}(\la^\natural + \ell \bf{1}).
\end{equation}

We first consider the limit case $n=\infty$ with the space of hamonics denoted by ${}^\Sp
H^\infty$. Here the only restriction on $\la$ is $\ell(\la)\le \ell$, and we observe that the
vector given in Theorem \ref{gldmn-hwv} associated to such a partition $\la$ is indeed
annihilated by $\mf{u}^+$ and hence is a joint $\text{Sp}(d)\times\gl(m|n)$-highest
weight vector of weight $(\la,\la^\natural+\ell\bf{1})$.  Hence all the summands on the
right hand side of (\ref{harmonic}) occur in the space of harmonics, and in particular, all
irreducible representations of $\Sp(d)$ occur. Therefore, we have established (\ref{harmonic}) in
this case.

Now consider the finite $n$ case. We may regard $S(\C^d\otimes\C^{m|n})\subseteq
S(\C^d\otimes\C^{m|\infty})$ with compatible actions
$\osp(2m|2n)\subseteq\osp(2m|2\infty)$.  From the formulas of the $\Delta$-operators in
(\ref{extra gen}) we see that ${}^\Sp H\subseteq{}^\Sp H^\infty$. Thus the space ${}^\Sp H$ is
obtained from ${}^\Sp H^\infty$ by setting the variables $\eta^i_k=0$, for $k>n$. However, it
is clear, from the explicit formulas of the joint highest vectors in ${}^\Sp H^\infty$, that
when setting the variables $\eta^i_k=0$, for $k>n$, precisely those vectors corresponding
to $(m|n)$-hook partitions will survive. This completes the proof.
\end{proof}

\subsection{The Howe dual pair $(\Sp(d), \ccc)$}

Set $d=2\ell$. Let $\mathbb{C}^{\infty}$ be the vector space over $\C$ with basis
$\{\,e_i\,|\,i\in\mathbb{Z}\,\}$.  Let $\gl_\infty$ denote the
Lie algebra consisting of matrices $(a_{ij})_{i,j\in\Z}$ with finitely many non-zero $a_{ij}$'s,
and $\C^\infty$ is the natural $\gl_\infty$-module. Then
$\gl_\infty$ has a linear basis given by the matrix units $E_{ij}$ ($i,j\in\Z$). Denote by
$\widehat{\gl}_\infty=\gl_\infty\oplus\C K$ the central extension of $\gl_\infty$ by a
one-dimensional center $\C K$ given by the $2$-cocycle
\begin{equation*}\label{cocycle}
\tau (A,B):={\rm Tr}([J,A]B),
\end{equation*}
where $J=\sum_{i\le 0}E_{ii}$. The Lie algebra $\widehat{\gl}_\infty$ admits a
$\Z$-grading $\widehat{\gl}_\infty =\bigoplus_{i \in \Z} \widehat{\gl}_{\infty,\pm i}$ by
letting $\deg E_{ij} =j-i$ and $\deg K=0$, and this induces a triangular decomposition of
$\widehat{\gl}_\infty$, with $\widehat{\gl}_{\infty,\pm} =\bigoplus_{i
>0} \widehat{\gl}_{\infty,\pm i}$.

\begin{xca} Show that the cocycle $\tau$ on $\gl_\infty$ is a coboundary and hence
$\widehat{\ga}_\infty$ is isomorphic to $\gl_\infty\oplus\C K$ as Lie algebras.
\end{xca}

Let $\overline{\mf{c}}_{\infty}$ be the subalgebra of ${\gl}_\infty$ preserving the
following bilinear form on $\mathbb{C}^{\infty}$:
\begin{equation*}
 (e_i|e_j)=
(-1)^i\delta_{i,1-j},  \quad i,j\in\Z.
\end{equation*}
Let $\ccc =\overline{\mf{c}}_{\infty}\oplus \mathbb{C}K$ be the central extension of
$\overline{\mf{c}}_{\infty}$ determined by the restriction of the two-cocycle $\tau$
above. Then $\ccc$ is an infinite-rank affine Kac-Moody algebra with Dynkin diagram as
follows:
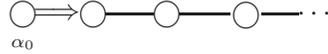
\begin{center}
\hskip -3cm \setlength{\unitlength}{0.16in}
\begin{picture}(24,3)
\put(5.7,2){\makebox(0,0)[c]{$\bigcirc$}} \put(8,2){\makebox(0,0)[c]{$\bigcirc$}}
\put(10.4,2){\makebox(0,0)[c]{$\bigcirc$}} \put(15.3,2){\makebox(0,0)[c]{$\cdots$}}
\put(6.8,2){\makebox(0,0)[c]{$\Longrightarrow$}}
\put(8.4,2){\line(1,0){1.55}} \put(10.82,2){\line(1,0){1.66}}
\put(13.5,2){\line(1,0){1.2}}
\put(13,1.95){\makebox(0,0)[c]{$\bigcirc$}} \put(5.7,1){\makebox(0,0)[c]{\tiny $\alpha_0$}}
\end{picture}
\end{center}
The Lie algebra $\mf{c}_{\infty}$ has a natural triangular decomposition induced from
$\widehat{\mathfrak{gl}}_{\infty}$:
\begin{equation*}
\mf{c}_{\infty}= \mf{c}_{\infty,-}\oplus \mathfrak{h} \oplus
\mf{c}_{\infty,+},
\end{equation*}
where the Cartan subalgebra of $\ccc$ is spanned by $K$ and
\begin{equation*}
\widetilde{E}_i := E_{ii}-E_{1-i,1-i}, \qquad i \in \N.
\end{equation*}
Let $\epsilon_i \in \mathfrak{h}^*$ be so that
$\langle\epsilon_i,\widetilde{E}_j\rangle=\delta_{ij}$, for $i,j\in \mathbb{N}$. For
$i\in \Z_+$, we denote by $\La^\mf{c}_i$ the $i$-th fundamental weight for $\ccc$. Then
$\La^\mf{c}_0\in \mf{h}^*$ is determined by
$\langle\La^\mf{c}_0,\widetilde{E}_i\rangle=0$ for $i\in\N$ and
$\langle\La^\mf{c}_0,K\rangle=1$, and
\begin{align*}
\La^{\mf c}_i&=\La^{\mf c}_0+\epsilon_1+\ldots+\epsilon_i, \qquad
\ i\geq 1.
\end{align*}
Let $L(\mathfrak{c}_{\infty}, \mu)$ denote the irreducible highest weight module with
respect to the Borel subalgebra $ \mathfrak{h} \oplus \mf{c}_{\infty ,+}$ of highest
weight $\mu$.

\begin{theorem} \cite[Theorem 3.4]{Wa}
There exists a Howe dual pair (${\rm Sp}(d),  \ccc$) acting on the
fermionic Fock space $\mf F^{\ell}$ of $\ell$ pairs of complex
fermions, where $d =2 \ell$. Moreover, as an $(\Sp(d),
\ccc)$-module,
\begin{equation} \label{eq:dual-c}
\mf F^{\ell}\cong\bigoplus_{\ell(\la) \le \ell} L({\rm Sp}(d), \la)
\otimes L(\mathfrak{c}_{\infty},\La^{\mf c}(\la)),
\end{equation}
where $\Lambda^{\mf c}(\la) :=\frac{d}{2}\La^{\mf c}_0+\sum_{k\geq 1}\la'_k\epsilon_k =
\sum_{k=1}^{\frac{d}{2}} \Lambda^{\mf c}_{\la_k}$.
\end{theorem}
As all we need later on is the combinatorial identity \eqref{combid-classical-c} below,
we will skip the precise definitions of the fermionic Fock space $\mf F^{\ell}$ and the
commuting actions of ${\rm Sp}(d)$ and $\ccc$ on $\mf F^{\ell}$  (see \cite{Wa} for
detail).

Recalling (\ref{gld}) we set $\widetilde{E}^{i}:={{E}^{ii}-E^{d-i+1,d-i+1}}$, for
$i=1,\ldots,\frac{d}{2}$. Computing the trace of $\prod_{n\in
\N}x_n^{\widetilde{E}_n}\prod_{i=1}^\frac{d}{2} z_i^{{\widetilde{E}^{i}}}$ on both sides
of \eqref{eq:dual-c}, we have
\begin{equation}\label{combid-classical-c}
\prod_{i=1}^\frac{d}{2}\prod_{n\in\N}(1+x_nz_i)(1+x_{n}z^{-1}_{i})=\sum_{\ell(\la)
\le \ell} {\rm ch}L({\rm Sp}(d), \la)\; {\rm
ch}L(\mathfrak{c}_{\infty},\Lambda^{\mf c}(\la)).
\end{equation}

\subsection{Irreducible $\ccc$-characters}

Let $W$ (respectively, $W_0$) be the Weyl group of $\ccc$ (respectively, of the Levi
subalgebra of $\ccc$ corresponding to the removal of the simple root $\alpha_0$). Let
$W^0_k$ be the set of the minimal length representatives of the right coset space
$W_0\backslash W$ of length $k$. Then it is standard to write $W =W_0W^0$ with $W^0
=\bigsqcup_{k \ge 0} W^0_k$.  It follows that, for each $w\in W^0_k$, we may find a
partition $\la_w=((\la_w)_1,(\la_w)_2,\ldots)$ such that
\begin{equation}\label{lambdaw}
w(\La^{\mf c}(\la)+\rho)-\rho= \frac{d}{2}\Lambda^{\mf
c}_0+\sum_{j>0}(\la_w)_j\epsilon_j.
\end{equation}
The following is obtained by applying the Kostant homology formula for integrable modules
over Kac-Moody algebras and the Euler-Poincar\'e principle (cf.~\cite[Proposition~2.5]{CKW}).

\begin{proposition} \label{char:schur}
We have  the following character formula:
\begin{equation*}
{\rm ch}L(\ccc,\Lambda^{\mathfrak c}(\la))=\frac{1}{ \prod_{1 \le
i\le j}(1-x_ix_j)}\sum_{k=0}^\infty \sum_{w\in W^0_k }
 (-1)^k s_{\la_w}(x_1,x_2,\ldots).
\end{equation*}
\end{proposition}

\begin{xca}
Prove the existence of $\la_w$ for each $w\in W^0$ in (\ref{lambdaw}).
\end{xca}
\subsection{The irreducible $\osp(2m|2n)$-characters}

The following character formula for $L(\osp(2m|2n),\la^\natural +
\ell \bf{1})$ in a different  form was first obtained in
\cite{CZ}. The proof here follows \cite{CKW} and is simpler.
\begin{theorem}\label{superchar:OSP}
For $\la\in\Pmn$ such that $\ell(\la) \le \ell$, we have
\begin{align*}
&{\rm ch}L(\osp(2m|2n),\la^\natural + \ell \bf{1}) =\\
& \left( \frac{y_1\cdots y_m}{x_1\cdots x_n} \right)^\ell
\frac{\prod\limits_{1\le i \le m}\prod\limits_{1\le s  \le
n}(1+y_ix_s)} {\prod\limits_{1\le i < j\le m}\prod\limits_{1\le s
\le t\le n}(1-y_iy_j)(1-x_sx_t)} \sum_{k=0}^\infty \sum_{w\in
W^0_k}(-1)^k hs_{\la_w}(\x,\y).
\end{align*}
\end{theorem}
\begin{proof}
Computing the trace of the operator
$\prod_{i,j}y_i^{\wt{E}_{i}}x_j^{\wt{E}_{\bar{j}}}\prod_{k=1}^{\frac{d}{2}}
z_k^{{\widetilde{E}^{k}}}$ on both sides of the isomorphism in Theorem~\ref{Sp-osp-Howe},
where $1\le i\le m$  and $1\le j\le n$, we obtain
\begin{align}\label{ops2m2n:totalchar}
& \prod_{k=1}^{\frac{d}{2}}\prod_{i,j}
\frac{(1+x_iz_k^{-1})(1+x_iz_k)}{(1-y_jz_k^{-1})(1-y_{j}z_k)}
  \nonumber  \\
= &\left( \frac{y_1\cdots y_m}{x_1\cdots x_n} \right)^{-\ell}
\sum_{\substack{\la\in \Pmn \\ \ell(\la) \le \ell}} {\rm ch}L({\rm
Sp}(d), \la)\; {\rm ch}L(\osp(2m|2n),\la^\natural + \ell \bf{1}).
\end{align}
Replacing $\text{ch}\, L(\ccc, \La^{\mf{c}}(\la))$ in \eqref{combid-classical-c} by the
expression in Proposition~\ref{char:schur}, we obtain an identity of symmetric functions
in variables $x_1, x_2,\ldots$. Next, we replace $x_{n+i}$ by $y_i$ for $i\in\N$, and
then apply to this identity the involution on the ring of symmetric functions in
$y_1,y_2,\ldots$.  Finally, we put $y_{j}=0$ for $j\geq m+1$. Under the composition of
those maps, we obtain a new identity which shares the same left hand side as
\eqref{ops2m2n:totalchar}. Comparing the right hand sides of this new identity and of
\eqref{ops2m2n:totalchar}, we obtain the result thanks to the linear independence of
${\rm ch}L({\rm Sp}(d), \la)$.
\end{proof}

\begin{remark} \label{rem:char}
The above irreducible character formula can be understood as an alternating sum of the
characters of parabolic Vermas over a certain infinite Weyl group. A similar observation
was made for type $A$ in \cite[Corollary 4.15]{CZ1}.

The Kostant $\mf u$-cohomology groups for these $\ga$-modules can be further computed and
described in terms of the infinite Weyl groups, and in particular, they are
multiplicity-free as modules over Levi subalgebras \cite{CKW} (also cf.~\cite{CK}). In
other words, the corresponding Kazhdan-Lusztig polynomials via $\mf u$-cohomology in the
sense of Vogan \cite{Vo} are monomials.
\end{remark}

\section{Super duality}
\label{sec:SD}

The goal of the remaining part of the lectures is to formulate precisely a direct
connection between representation theories of Lie algebras and Lie superalgebras,
following \cite{CLW} (cf. \cite{CWZ, CW4, CL}). This is in part supported by the
following.
\begin{enumerate}
\item The $\gl(1|1)$-principal block can be understood in terms of
    $A_\infty$-quivers, which is classical in its own way.

\item There exist intimate connections between finite-dimensional classical Lie
    superalgebras and infinite Weyl groups (cf. Remark~\ref{rem:char}).

\item By a purely algebraic approach \cite{Br}, Brundan showed
that
    the category of finite-dimensional $\glmn$-modules categorifies the Fock space
    $\wedge^m \mathbb V \otimes \wedge^n \mathbb V^*$, where $\mathbb V$ denotes the
    natural $U_q(\gl (\infty))$-module.

\item The Brundan-Kazhdan-Lusztig polynomials for some suitable parabolic category $\mc O$ of
    super type $A$ are shown in \cite{CWZ, CW4} to coincide with the usual
    Kazhdan-Lusztig polynomials of type $A$. For the so-called polynomial
    representations this was already observed in \cite{CZ1}. Moreover, the super
    duality conjecture \cite{CW4} (generalizing \cite{CWZ}) on a category equivalence
    between Lie superalgebras and Lie algebras of type $A$ has been established
    recently in \cite{CL}.
\end{enumerate}
\subsection{Parabolic category $\mc O$}
\label{sec:O}

Assume that $\mf k$ is a (possibly Kac-Moody, and or possibly infinite rank) Lie
(super)algebra, with Cartan subalgebra $\h$ and root system $\Phi$. Let $I$ be a set of
simple roots, and $\mf b$ the corresponding Borel subalgebra. Let $Y$ be a subset of $I$
and $\mf l_Y$ be the associated Levi subalgebra of $\mf k$. Let $\mf k = \mf u^- \oplus
\mf l_Y \oplus \mf u^+$ be the generalized triangular decomposition. Denote by $P$ the
weight lattice for $\mf k$, and by $P_Y \subseteq P$ the subset of weights which are
$Y$-dominant.

Given $\la\in{\h}^*$, we denote by $L({\mf{l}}_Y,\la)$ the irreducible
${\mf{l}}_Y$-module of highest weight $\la$ with respect to $\mf{l}_Y\cap\mf{b}$, which
extends to a $(\mf l_Y + \mf u^+)$-module with a trivial action of $\mf u^+$. Define the
parabolic Verma $\mf k$-module
\begin{equation*}
\Delta (\la):= U(\mf k) \otimes_{U (\mf l_Y + \mf u^+)}
L({\mf{l}}_Y,\la).
\end{equation*}
Let ${L}(\la)$ be the irreducible quotient $\mf k$-module of $\Delta (\la)$.

Let $\mc{O}=\mc O(\mf k, Y, P_Y)$ be the category of $\ga$-modules ${M}$ such that ${M}$
is a semisimple ${\h}$-module with finite-dimensional weight subspaces $M_\gamma$,
$\gamma\in {\h}^*$, satisfying
\begin{itemize}
\item[(i)] ${M}$ decomposes over ${\mf{l}_Y}$ into a direct sum of
    $L({\mf{l}_Y},\mu)$ for $\mu\in {P}_Y$.

\item[(ii)] There exist finitely many weights $\la_1,\la_2,\ldots,\la_k\in{P}_Y$
    (depending on ${M}$) such that if $\gamma$ is a weight in ${M}$, then
    $\gamma\in\la_i-\sum_{\alpha\in{I}}\Z_+\alpha$, for some $i$.
\end{itemize}
The choice of $P_Y$ should be natural and general enough so that
$L(\la), \Delta(\la) \in \mc O(\mf k, Y, P_Y)$ for every $\la \in
P_Y$. This will be spelled out clearly in the cases of interest
later on.

\begin{remark}  \label{KL}
In the case when $\mf k$ is a finite-dimensional semisimple Lie algebra, or $\gl (n)$, or
of the four classical infinite-rank affine Kac-Moody Lie algebras $\aaa$, $\bbb$, $\ccc$,
$\ddd$, each block is controlled by the Weyl group $W$ of $\mf k$. For the principal
block (in the full category $\mc O$), the transition matrix between $[L(\la)]$ and
$[\Delta (\mu)]$ is given by the Kazhdan-Lusztig polynomials associated to $W$ (see
Tanisaki's lectures \cite{Ta}). Other blocks and parabolic cases are reduced to the
principal block.
\end{remark}

\subsection{The master Lie (super)algebras}\label{masters}

Let \makebox(23,0){$\oval(20,12)$}\makebox(-20,8){$\mf{x}$} denote
one of the following four classical Dynkin diagrams of types
$\mf{a,b,c,d}$, which will be referred to as a {\em head diagram}:
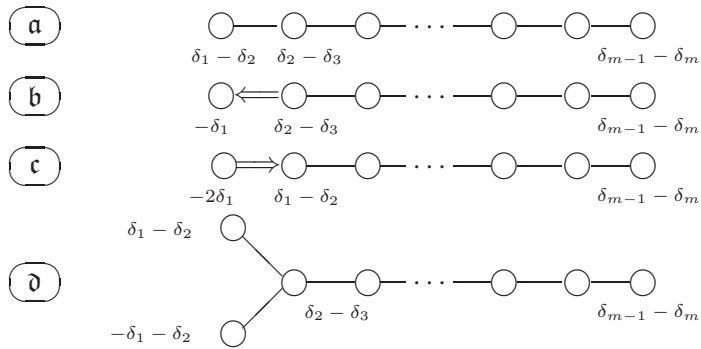
\begin{center}
\hskip -1cm \setlength{\unitlength}{0.16in}
\begin{picture}(24,3)
\put(8,2){\makebox(0,0)[c]{$\bigcirc$}} \put(10.4,2){\makebox(0,0)[c]{$\bigcirc$}}
\put(14.85,2){\makebox(0,0)[c]{$\bigcirc$}} \put(17.25,2){\makebox(0,0)[c]{$\bigcirc$}}
\put(19.4,2){\makebox(0,0)[c]{$\bigcirc$}} \put(5.6,2){\makebox(0,0)[c]{$\bigcirc$}}
\put(8.4,2){\line(1,0){1.55}} \put(10.82,2){\line(1,0){0.8}}
\put(13.2,2){\line(1,0){1.2}} \put(15.28,2){\line(1,0){1.45}}
\put(17.7,2){\line(1,0){1.25}} \put(6,2){\line(1,0){1.4}}
\put(12.5,1.95){\makebox(0,0)[c]{$\cdots$}}
\put(-.5,2){\makebox(0,0)[c]{{\ovalBox(1.6,1.2){$\mf{a}$}}}}
\put(5.7,1){\makebox(0,0)[c]{\tiny$\delta_1-\delta_{2}$}}
\put(8.5,1){\makebox(0,0)[c]{\tiny$\delta_{2}-\delta_{3}$}}
\put(19.6,1){\makebox(0,0)[c]{\tiny$\delta_{m-1}-\delta_{m}$}}
\end{picture}
\end{center}
\begin{center}
\hskip -1cm \setlength{\unitlength}{0.16in}
\begin{picture}(24,2)
\put(8,2){\makebox(0,0)[c]{$\bigcirc$}} \put(10.4,2){\makebox(0,0)[c]{$\bigcirc$}}
\put(14.85,2){\makebox(0,0)[c]{$\bigcirc$}} \put(17.25,2){\makebox(0,0)[c]{$\bigcirc$}}
\put(19.4,2){\makebox(0,0)[c]{$\bigcirc$}} \put(5.6,2){\makebox(0,0)[c]{$\bigcirc$}}
\put(8.4,2){\line(1,0){1.55}} \put(10.82,2){\line(1,0){0.8}}
\put(13.2,2){\line(1,0){1.2}} \put(15.28,2){\line(1,0){1.45}}
\put(17.7,2){\line(1,0){1.25}} \put(6,1.8){$\Longleftarrow$}
\put(12.5,1.95){\makebox(0,0)[c]{$\cdots$}}
\put(-.5,2){\makebox(0,0)[c]{{\ovalBox(1.6,1.2){$\mf{b}$}}}}
\put(5.3,1){\makebox(0,0)[c]{\tiny$-\delta_1$}}
\put(8.4,1){\makebox(0,0)[c]{\tiny$\delta_2-\delta_{3}$}}
\put(19.6,1){\makebox(0,0)[c]{\tiny$\delta_{m-1}-\delta_{m}$}}
\end{picture}
\end{center}

\begin{center}
\hskip -1cm \setlength{\unitlength}{0.16in}
\begin{picture}(24,2)
\put(5.7,2){\makebox(0,0)[c]{$\bigcirc$}} \put(8,2){\makebox(0,0)[c]{$\bigcirc$}}
\put(10.4,2){\makebox(0,0)[c]{$\bigcirc$}} \put(14.85,2){\makebox(0,0)[c]{$\bigcirc$}}
\put(17.25,2){\makebox(0,0)[c]{$\bigcirc$}} \put(19.4,2){\makebox(0,0)[c]{$\bigcirc$}}
\put(6.8,2){\makebox(0,0)[c]{$\Longrightarrow$}} \put(8.4,2){\line(1,0){1.55}}
\put(10.82,2){\line(1,0){0.8}} \put(13.2,2){\line(1,0){1.2}}
\put(15.28,2){\line(1,0){1.45}} \put(17.7,2){\line(1,0){1.25}}
\put(12.5,1.95){\makebox(0,0)[c]{$\cdots$}}
\put(-.5,2){\makebox(0,0)[c]{{\ovalBox(1.6,1.2){$\mf{c}$}}}}
\put(5.3,1){\makebox(0,0)[c]{\tiny$-2\delta_1$}}
\put(8.4,1){\makebox(0,0)[c]{\tiny$\delta_1-\delta_{2}$}}
\put(19.6,1){\makebox(0,0)[c]{\tiny$\delta_{m-1}-\delta_{m}$}}
\end{picture}
\end{center}
\begin{center}
\hskip -1cm \setlength{\unitlength}{0.16in}
\begin{picture}(24,3.5)
\put(8,2){\makebox(0,0)[c]{$\bigcirc$}} \put(10.4,2){\makebox(0,0)[c]{$\bigcirc$}}
\put(14.85,2){\makebox(0,0)[c]{$\bigcirc$}} \put(17.25,2){\makebox(0,0)[c]{$\bigcirc$}}
\put(19.4,2){\makebox(0,0)[c]{$\bigcirc$}} \put(6,3.8){\makebox(0,0)[c]{$\bigcirc$}}
\put(6,.3){\makebox(0,0)[c]{$\bigcirc$}} \put(8.4,2){\line(1,0){1.55}}
\put(10.82,2){\line(1,0){0.8}} \put(13.2,2){\line(1,0){1.2}}
\put(15.28,2){\line(1,0){1.45}} \put(17.7,2){\line(1,0){1.25}}
\put(7.6,2.2){\line(-1,1){1.3}} \put(7.6,1.8){\line(-1,-1){1.3}}
\put(12.5,1.95){\makebox(0,0)[c]{$\cdots$}}
\put(-.5,2){\makebox(0,0)[c]{{\ovalBox(1.6,1.2){$\mf{d}$}}}}
\put(3.3,0.3){\makebox(0,0)[c]{\tiny$-\delta_1-\delta_{2}$}}
\put(3.6,3.7){\makebox(0,0)[c]{\tiny$\delta_1-\delta_{2}$}}
\put(9.4,1){\makebox(0,0)[c]{\tiny$\delta_{2}-\delta_{3}$}}
\put(19.6,1){\makebox(0,0)[c]{\tiny$\delta_{m-1}-\delta_{m}$}}
\end{picture}
\end{center}

A head diagram
\makebox(23,0){$\oval(20,12)$}\makebox(-20,8){$\mf{x}$} is
connected with one of the three tail diagrams to produce three
``master" Dynkin diagrams as below, whose corresponding Lie
(super)algebras will be denoted by $\mathcal G =\mathcal G^{\mf
x}$, $\ov {\mathcal G} =\ov {\mathcal G}^{\mf x}$, and $\wt
{\mathcal G} =\wt {\mathcal G}^{\mf x}$, respectively.

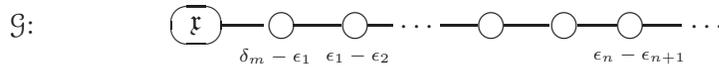
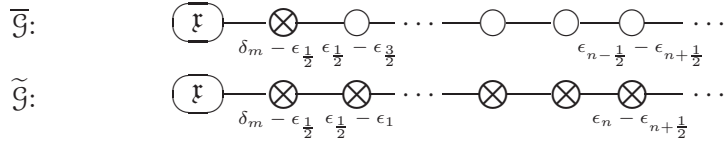
\begin{center}
\hskip -1cm \setlength{\unitlength}{0.16in}
\begin{picture}(24,3)
\put(8,2){\makebox(0,0)[c]{$\bigcirc$}}
\put(10.4,2){\makebox(0,0)[c]{$\bigcirc$}}
\put(14.85,2){\makebox(0,0)[c]{$\bigcirc$}}
\put(17.25,2){\makebox(0,0)[c]{$\bigcirc$}}
\put(19.4,2){\makebox(0,0)[c]{$\bigcirc$}}
\put(5.2,2){\makebox(0,0)[c]{{\ovalBox(1.6,1.2){$\mf{x}$}}}}
\put(8.4,2){\line(1,0){1.55}} \put(10.82,2){\line(1,0){0.8}}
\put(13.2,2){\line(1,0){1.2}} \put(15.28,2){\line(1,0){1.45}}
\put(17.7,2){\line(1,0){1.25}} \put(19.8,2){\line(1,0){1.25}}
\put(6,2){\line(1,0){1.4}}
\put(12.5,1.95){\makebox(0,0)[c]{$\cdots$}}
\put(22,1.95){\makebox(0,0)[c]{$\cdots$}}
\put(-.5,2){\makebox(0,0)[c]{${\mathcal G}$:}}
\put(7.8,1){\makebox(0,0)[c]{\tiny$\delta_m-\epsilon_1$}}
\put(10.5,1){\makebox(0,0)[c]{\tiny$\epsilon_1-\epsilon_2$}}
\put(19.7,1){\makebox(0,0)[c]{\tiny$\epsilon_n-\epsilon_{n+1}$}}
\end{picture}
\end{center}
\begin{center}
\hskip -1cm \setlength{\unitlength}{0.16in}
\begin{picture}(24,2)
\put(8,2){\makebox(0,0)[c]{$\bigotimes$}}
\put(10.4,2){\makebox(0,0)[c]{$\bigcirc$}}
\put(14.85,2){\makebox(0,0)[c]{$\bigcirc$}}
\put(17.25,2){\makebox(0,0)[c]{$\bigcirc$}}
\put(19.4,2){\makebox(0,0)[c]{$\bigcirc$}}
\put(5.2,2){\makebox(0,0)[c]{{\ovalBox(1.6,1.2){$\mf{x}$}}}}
\put(8.4,2){\line(1,0){1.55}} \put(10.82,2){\line(1,0){0.8}}
\put(13.2,2){\line(1,0){1.2}} \put(15.28,2){\line(1,0){1.45}}
\put(17.7,2){\line(1,0){1.25}} \put(19.8,2){\line(1,0){1.25}}
\put(6,2){\line(1,0){1.4}}
\put(12.5,1.95){\makebox(0,0)[c]{$\cdots$}}
\put(22,1.95){\makebox(0,0)[c]{$\cdots$}}
\put(-.5,2){\makebox(0,0)[c]{$\ov{{\mathcal G}}$:}}
\put(7.8,1){\makebox(0,0)[c]{\tiny$\delta_m-\epsilon_{\hf}$}}
\put(10.5,1){\makebox(0,0)[c]{\tiny$\epsilon_{\hf}-\epsilon_{\frac{3}{2}}$}}
\put(19.7,1){\makebox(0,0)[c]{\tiny$\epsilon_{n-\hf}-\epsilon_{n+\hf}$}}
\end{picture}
\end{center}
\begin{center}
\hskip -1cm \setlength{\unitlength}{0.16in}
\begin{picture}(24,2)
\put(8,2){\makebox(0,0)[c]{$\bigotimes$}} \put(10.4,2){\makebox(0,0)[c]{$\bigotimes$}}
\put(14.85,2){\makebox(0,0)[c]{$\bigotimes$}}
\put(17.25,2){\makebox(0,0)[c]{$\bigotimes$}}
\put(19.4,2){\makebox(0,0)[c]{$\bigotimes$}}
\put(5.2,2){\makebox(0,0)[c]{{\ovalBox(1.6,1.2){$\mf{x}$}}}}
\put(8.4,2){\line(1,0){1.55}} \put(10.82,2){\line(1,0){0.8}}
\put(13.2,2){\line(1,0){1.2}} \put(15.28,2){\line(1,0){1.45}}
\put(17.7,2){\line(1,0){1.25}} \put(19.8,2){\line(1,0){1.25}}
\put(6,2){\line(1,0){1.4}}
\put(12.5,1.95){\makebox(0,0)[c]{$\cdots$}}
\put(22,1.95){\makebox(0,0)[c]{$\cdots$}}
\put(-.5,2){\makebox(0,0)[c]{$\wt{{\mathcal G}}$:}}
\put(7.8,1){\makebox(0,0)[c]{\tiny$\delta_m-\epsilon_\hf$}}
\put(10.5,1){\makebox(0,0)[c]{\tiny$\epsilon_\hf-\epsilon_1$}}
\put(19.7,1){\makebox(0,0)[c]{\tiny$\epsilon_n-\epsilon_{n+\hf}$}}
\end{picture}
\end{center}For example, the resulting
master Dynkin diagrams for ${\mathcal G}$ are the four
infinite-rank classical Lie algebras $\aaa$ (one-sided infinity),
$\bbb, \ccc, \ddd$.

%
%
\subsection{The categories $\mc O, \ov{\mc O}, \wt{\mc O}$}\label{sec:allO}

Let $\Pi^{\mf x}$ denote the standard set of simple roots associated to the diagram
\makebox(23,0){$\oval(20,12)$}\makebox(-20,8){$\mf{x}$}, and fix $Y^0 \subset \Pi^{\mf
x}$. Let $\mf T^1$, $\ov{\mf{T}}^1$ and $\wt{\mf{T}}^1$ be the subsets of the three tail
diagrams
with the first simple root removed. Let
$$
Y = Y^0 \cup \mf T^1,
 \qquad \ov Y =Y^0 \cup \ov{\mf T}^1,
 \qquad \wt Y =Y^0 \cup \wt{\mf T}^1.
$$
Let $P$ (respectively, $\ov P$, $\wt P$) be the lattice which consists of the integral
span of $\ep_i$, where $i$ runs over $I(m|\N)$ (respectively, $I(m|\hf +\N)$, $I(m|\hf
\N)$). Here $I(m|\N) =\{\bar{1}, \ldots, \bar{m}\} \cup \N$, and the other two sets are
similarly defined.

 Define $P_Y = \cup_{d \in \Z} P_Y^d$, where
\begin{align}
P_Y^d  =\Big\{ \la = & \sum_{i \in I(m|\N)} \la_i \ep_i \in P \mid
\la\ \text{ is } Y
\text{-dominant};
\text{ } \la_i =d  \text{ for } i\gg 0 \Big\}.
\end{align}
We define ${P}_{\ov Y}^d$ by changing $I(m|\N)$ in the definition of $P_Y^d$ to
$I(m|\hf+\N)$ and changing $ \la_i =d$ to $ \la_i =-d$ for $i\gg 0$, and let $P_{\ov Y} = \cup_{d \in
\Z} {P}_{\ov Y}^d$.

Then, for $\la \in P_Y^d$,  $\la^+ :=(\la_1 -d, \la_2 -d, \ldots)$ is a partition. Also,
for $\la \in P_{\ov Y}^d$,  $\la^+ :=(\la_{1/2} +d, \la_{3/2} +d, \ldots)$ is a
partition. We have a bijection of sets:
$$
\natural: P^d_Y \longrightarrow P^d_{\ov Y},    \qquad \la \mapsto
\la^\natural,
$$
which is determined by the conditions
$$
\la_i= \la_i^\natural, \quad \forall i \in I(m|0); \qquad
(\la^\natural)^+ = (\la^+)'.
$$
This further induces a bijection $\natural: P_Y \rightarrow P_{\ov Y}$.

We define another map
$$
\theta: P_Y^d \rightarrow P_{\wt Y}^d, \quad
 \la \mapsto \la^\theta =\sum_{i \in I (m|0)} \la_i\ep_i +\sum_{i
\in \N} \big( (p_i -d) \ep_{i-\hf} +(q_i+d) \ep_i \big),
$$
where the image $\la^\theta$ is determined by the conditions
$\la_i= \la_i^\theta$ for all $i \in I(m|0)$ and $(p_i|q_i)_{i \ge
1}$ is the modified Frobenius coordinates of the partition
$(\la^+)'$ (as defined in Example~\ref{ex:3hw}~(3)). The map
$\theta$ is clearly injective, and so we obtain a bijection
$\theta: P_Y^d \rightarrow P_{\wt Y}^d :=\theta (P_Y^d)$ for each
$d \in \Z$. Letting $P_{\wt Y} =\cup_{d \in \Z} P_{\wt Y}^d$ we
also obtain a bijection $\theta: P_Y \rightarrow P_{\wt Y}$.

Following the construction of a parabolic category $\mc O$ in Section~\ref{sec:O}, we
obtain the categories $\mc O :=\mc O({\mathcal G}, Y, P_Y), \ov{\mc O} :={\mc O}(\ov
{\mathcal G}, \ov Y, P_{\ov Y})$, and $\wt{\mc O} :={\mc O}(\wt {\mathcal G}, \wt Y,
P_{\wt Y})$. We observe the following direct sum decompositions of categories:
$$
 \mc O = \bigoplus_{d \in \Z} {\mc O}^d, \quad
 \ov{\mc O} = \bigoplus_{d \in \Z} \ov{\mc O}^d, \quad
 \wt{\mc O} = \bigoplus_{d \in \Z} \wt{\mc O}^d,
$$
where the index $d$ indicates additional weight constraints: $(\la, \ep_i) =d$ for $i \in
\Z, i\gg 0$ for $\la \in P_Y$; $(\la, \ep_i) =-d$ for $i \in \hf +\Z, i\gg 0$ for $\la
\in P_{\ov Y}$; $(\la, \ep_i) =(-1)^{2i}d$ for $i \in \hf \Z, i\gg 0$ for $\la \in P_{\wt
Y}$, respectively.

\begin{remark}
A more conceptual formulation (see \cite{CLW} for details) is to
introduce central extensions of ${\mathcal G}, \ov {\mathcal G},
\wt {\mathcal G}$ respectively by a one-dimensional center, whose
module categories at level $d$ are equivalent to ${\mc O}^d,
\ov{\mc O}^d,  \wt{\mc O}^d$ above, respectively. In this way, the
additional weight constraints above are transferred to be $(\la,
\ep_i) =0$ for $i\gg 0$, which is independent of $d$.
\end{remark}

\subsection{The equivalence of categories}\label{Tfunctors}

We start with two simple observations:

(i) The Lie (super)algebras ${\mathcal G}$ and $\ov {\mathcal G}$
are naturally Lie subalgebras (though not Levi subalgebras) of
$\wt {\mathcal G}$, and the standard triangular decompositions of
the three Lie (super)algebras are compatible with the inclusions
${\mathcal G} \subset \wt {\mathcal G}$ and $\ov {\mathcal G}
\subset \wt {\mathcal G}$. This follows by examining the simple
roots of the three algebras.

(ii) We may naturally regard the lattices $P$ and $\ov P$ as sublattices of $\wt P$, by
definition of the lattices $P, \ov P, \wt P$.

Given a $\wt {\mathcal G}$-module with weight space decomposition
$\wt{M}=\bigoplus_{\mu\in\wt{P}}\wt{M}_\mu$, we define
\begin{align*}
T(\wt{M}):= \bigoplus_{\mu\in{P}}\wt{M}_\mu,\qquad
\hbox{and}
 \qquad
\ov{T}(\wt{M}):= \bigoplus_{\mu\in{\ov{P}}}\wt{M}_\mu.
\end{align*}
Clearly, $T(\wt{M})$ is a ${\mathcal G}$-module and
$\ov{T}(\wt{M})$ is a $\ov {\mathcal G}$-module. We can check that
$T$ and $\ov T$ send objects in category $\wt{\mc O}$ to objects
in category $\mc O$ and $\ov{\mc O}$, respectively, and hence we
obtain functors
\begin{eqnarray*}
\CD
 \wt{\mc{O}} @>\wt{T}>>
{\mc{O}} \\
 @VV{\ov{T}}V   \\
 \ov{\mc{O}}
 \endCD
\end{eqnarray*}
More precisely, we have  $T:\wt{\mc{O}}^d \rightarrow{\mc{O}}^d$ and
$\ov{T}:\wt{\mc{O}}^d \rightarrow\ov{\mc{O}}^d$.

We will use $\ov{\Delta}(\la)$ and $\wt{\Delta}(\la)$ to denote the
parabolic Verma modules in the categories $\ov{\mc O}$ and $\widetilde{\mc O}$ of highest
weights $\la$, respectively. Similarly, we use
$\ov{L}(\la)$ and $\wt{L}(\la)$ to denote the corresponding irreducible
modules. We are ready to formulate the main result of \cite{CLW}.

\begin{theorem} \cite{CLW} \label{th:SD}
We have
\begin{align*}
 T\big{(}\wt{\Delta}(\la^\theta)\big{)} =\Delta(\la), &\qquad
\ov{T}\big{(}\wt{\Delta}(\la^\theta)\big{)}=\ov{\Delta}(\la^\natural),   \\
 T\big{(}\wt{L}(\la^\theta)\big{)} =L(\la), &\qquad
\ov{T}\big{(}\wt{L}(\la^\theta)\big{)}=\ov{L}(\la^\natural).
\end{align*}
Moreover, $T:\wt{\mc{O}}\rightarrow{\mc{O}}$ and
$\ov{T}:\wt{\mc{O}}\rightarrow\ov{\mc{O}}$ are equivalences of categories.
\end{theorem}
As a consequence, $T:\wt{\mc{O}}^d\rightarrow{\mc{O}}^d$ and
$\ov{T}:\wt{\mc{O}}^d\rightarrow\ov{\mc{O}}^d$ are equivalences of categories.

\begin{proof}[Idea of a proof]
The proof consists of several major steps:

(i) $T(L(\mf l_{\wt Y}, \la^\theta)) =L(\mf l_{Y}, \la)$. This follows by locating a
highest weight $\la$ for the module $L(\mf l_{\wt Y}, \la^\theta)$ with respect to a new
Borel of $\mf l_{\wt{Y}}$ which is compatible with the imbedding $\mf l_{Y} \subseteq \mf
l_{\wt{Y}}$. Then we check the equality on the character level (which boils down to the
hook-Schur functions.)

(ii) $T(\wt{\Delta}(\la^\theta))$ is a highest weight module of
highest weight $\la$ with respect to the standard Borel of
${\mathcal G}$. This follows from a weight argument.

(iii) $T\big{(}\wt{\Delta}(\la^\theta)\big{)} =\Delta(\la)$. Check on the character
level, and use (i) and (ii).

(iv) $T\big{(}\wt{L}(\la^\theta)\big{)} =L(\la)$. Follows by showing that there is no
singular vector in $T\big{(}\wt{L}(\la^\theta)\big{)}$ of weight different
from $\la$.

This proves the  equality for $T$ in the theorem. The proof of the equality for $\ov{T}$
is similar. More work is needed to establish the equivalence of categories.
\end{proof}

\begin{corollary}  [Super Duality]
The categories $\mc O$ and $\ov{\mc O}$ are equivalent.
\end{corollary}

\begin{remark} [History of Super Duality]

First for type $A$.
\begin{enumerate}
\item[(1)] For $Y^0 =\Pi^\mf{a}$, the super duality was conjectured in
    \cite[Conjecture 6.10]{CWZ}. This was motivated by Brundan \cite{Br} on
    finite-dimensional simple $\glmn$-characters.

\item[(2)] For arbitrary Levi subalgebra associated to $Y^0 \subseteq \Pi^{\mf a}$,
    the super duality conjecture was formulated in \cite[Conjecture 4.18]{CW4}. It is
    shown in \cite{CWZ, CW4} that the Brundan-Kazhdan-Lusztig polynomials coincide
    with the usual Kazhdan-Lusztig polynomials in type $A$.

\item[(3)] The conjecture of \cite{CWZ} is proved in \cite{BrS}, where the underlying
    algebras for the two categories were computed and shown to be isomorphic.

\item[(4)] The more general conjecture of \cite{CW4} is independently proved in
    \cite{CL} essentially in the form of Theorem~\ref{th:SD}. In particular, this
    provides a new proof of the main result of \cite{Br}.

\end{enumerate}

Then for type $\osp$.
\begin{enumerate}
\item[(5)] A first supporting evidence for the super duality was provided by the
    computation in \cite {CKW} of the Kostant $\mf u$-homology groups with
    coefficients in the modules of classical Lie superalgebras appearing in Howe
   duality decompositions.

\item[(6)] Theorem~\ref{th:SD} is formulated and established in \cite{CLW}.
\end{enumerate}

Note that little has been known about representation theory of
infinite-dimensional Lie superalgebras, such as affine
superalgebras, with the exception of work of Kac and Wakimoto
\cite{KW}.
\begin{enumerate}
\item[(7)] The super duality approach applies to and sheds new insight on more
    general Lie superalgebras, including affine superalgebras.
\end{enumerate}
\end{remark}

\subsection{Irreducible $\ov{{\mathcal
G}}$-characters}\label{char}

By Remark~\ref{KL}, the irreducible character problem for the category $\mc O$ is solved
by the Kazhdan-Lusztig conjectures \cite{KL} (theorem of Beilinson-Bernstein \cite{BB}
and Brylinski-Kashiwara \cite{BK}). Hence by Theorem~\ref{th:SD}, the irreducible
character problem for the module category $\ov{\mc O}$ for Lie superalgebras is solved by
the same classical Kazhdan-Lusztig polynomials.

Recall for $\mf x \in \{\mf{a,b,c,d}\}$, we have defined a Lie superalgebra $\ov{\mathcal
G} =\ov{\mathcal G}^{\mf x}$ whose Dynkin diagram is obtained by connecting
\makebox(23,0){$\oval(20,12)$}\makebox(-20,8){$\mf{x}$} with the type $\gl(1|\infty)$
tail diagram $\ov{\mf T}$. For  $n \in \N$, we will also consider a Lie superalgebra
$\ov{\mathcal G}_n =\ov{\mathcal G}^{\mf x}_n$ whose Dynkin diagram (see below) is
obtained by connecting \makebox(23,0){$\oval(20,12)$}\makebox(-20,8){$\mf{x}$} with a
tail diagram $\ov{\mf{T}}_n$ of type $\gl(1|n)$:
\begin{equation}\label{finite diagram}
\hskip -6cm \setlength{\unitlength}{0.16in}
\begin{picture}(24,2)
\put(12.35,1){\makebox(0,0)[c]{{\ovalBox(1.6,1.2){$\mf{x}$}}}}
\put(15.25,1){\makebox(0,0)[c]{$\bigotimes$}}
\put(17.4,1){\makebox(0,0)[c]{$\bigcirc$}}
\put(21.9,1){\makebox(0,0)[c]{$\bigcirc$}}
\put(13.28,1){\line(1,0){1.45}} \put(15.7,1){\line(1,0){1.25}}
\put(17.8,1){\line(1,0){0.9}} \put(20.1,1){\line(1,0){1.4}}
\put(19.6,0.95){\makebox(0,0)[c]{$\cdots$}}
\put(15,0){\makebox(0,0)[c]{\tiny $\delta_m -\ep_{\hf}$}}
\put(18,0){\makebox(0,0)[c]{\tiny $\ep_\hf -\ep_{\frac32}$}}
\put(23,0){\makebox(0,0)[c]{\tiny $\ep_{n-\frac32} -\ep_{n-\hf}$}}
\put(7.5,1){\makebox(0,0)[c]{{\ovalBox(1.6,1.2){$\ov{\mathcal
G}^\mf{x}_n$}}}} \put(8.5,0.7){:}
\end{picture}
\end{equation}
As \makebox(23,0){$\oval(20,12)$}\makebox(-20,8){$\mf{x}$} is
taken to be one of the four classical Dynkin diagrams, $\ov
{\mathcal G}_n$ is a classical finite-dimensional Lie superalgebra
of type either $\gl$ or $\osp$.

Let $\ov{\mf{T}}_n^1$ be the subset of  $\ov{\mf{T}}_n$, which is obtained from
$\ov{\mf{T}}_n$ by the removal of the first simple root. Let $\ov Y_n =Y^0 \cup
\ov{\mf{T}}_n^1$, and let
\begin{align*}
P_{\ov Y_n}^d =\Big\{ \la = & \sum_{i \in I(m|n-\hf)} \la_i \ep_i
\mid \la \text{ is } \ov{Y}_n\text{-dominant; } \la_{n-\hf}
\geq -d  \Big\}.
\end{align*}
For $d \in \Z$, we introduce a category $\ov{\mc O}_n^d :=\mc O
(\ov {\mathcal G}_n, \ov Y_n, P_{\ov Y_n}^d)$ of $\ov {\mathcal
G}_n$-modules. Note that $\ov{\mc O}_n^d$ is a full subcategory of
the (more standard) category $\ov{\mc O}_n :=\mc O (\ov {\mathcal
G}_n, \ov Y_n, P_{\ov Y_n})$, where
\begin{align*}
P_{\ov Y_n} =\Big\{ \la = & \sum_{i \in I(m|n-\hf)} \la_i \ep_i
\mid \la \text{ is } \ov{Y}_n\text{-dominant} \Big\}.
\end{align*}
Moreover, $\ov{\mc O}_n^d \subset \ov{\mc O}_n^{d+1}$ for all $d$, and $\ov{\mc O}_n
=\cup_{d \in \Z} \ov{\mc O}_n^d.$

We also introduce truncation functors $\ov \Tr_{n}: \ov{\mc O}^d \rightarrow \ov{\mc
O}_n^d$ as follows: Let $M \in \ov{\mc O}^d$ such that $M =\bigoplus_\mu M_\mu$. We define
$$
\ov \Tr_{n}(M) =\bigoplus_{\{\mu \mid \mu_{n+\hf} =-d\} } M_\mu,
$$
which is clearly a $\ov {\mathcal G}_n$-module. Note that
$\mu_{n+\hf} =-d$ ensures $\mu_{k+\hf} =-d$ for all $k \geq n$.
Such a truncation functor to a Levi subalgebra (except that we
drop an abelian direct summand of the Levi) has its analogue in
algebraic group setting studied by Donkin (cf.~e.g.~\cite{CW4} for
a formulation in type $A$). The functors $\ov \Tr_{n}:\ov{\mc
O}^d\rightarrow\ov{\mc O}^d_n$ have the following key property.

\begin{proposition}\label{finite trunc}
For $X=\Delta$ or $L$, we have
$$
\ov \Tr_{n} (X (\la)) = \left \{
 \begin{array}{ll}
X_n (\la^{(n)}), & \text{ if } \la_{n+\hf} =-d, \\
0, & \text{ otherwise,}
\end{array}
\right.
$$
where $\la^{(n)}$ is obtained from $\la$ by ignoring $\la_{j}$, for $j>n-\hf$.
\end{proposition}

\begin{remark}
By Proposition \ref{finite trunc} and Theorem
\ref{th:SD}, the classical Kazhdan-Lusztig solution of the
irreducible character problem of $\mc O$ induces a complete solution of the irreducible
character problem of $\ov{\mc O}_n$ for each finite $n$.
It is not hard to show \cite{CLW} that every finite-dimensional irreducible modules of
every ortho-symplectic Lie superalgebra appears in some $\ov{\mc O}_n$.
\end{remark}

We end the paper with a list of symbols.

\medskip

\label{table1}
\begin{tabular}{r l}
{\bf Symbol}& {\bf Meaning}\\
\hline\hline\\
$|v|$& $\Z_2$-parity of a homogeneous vector $v$ in a vector superspace\\
$\C^{m|n}$&  complex vector superspace of super dimension $m|n$\\
$\gl (m|n)$&  the general Lie superalgebra \\
$I(m|n)$&  the totally order set $\{\ov 1<\ov 2< \ldots<\ov m< 1<\ldots<n\}$\\
$\sgl (m|n)$&  the special linear Lie superalgebra \\
$\h$&  Cartan subalgebra (of diagonal matrices)\\
$\delta_i$&  dual basis element corresponding to $E_{\ov i\ov i}$\\
$\ep_j$&  dual basis element corresponding to $E_{jj}$\\
$\Phi, \Phi_{\bar 0}, \text{ or } \Phi_{\bar 1}$&  sets of all, even, or odd
    roots\\
$\Pi$& set of simple roots\\
$\Phi^\pm_\epsilon$&  set of positive/negative roots of parity
    $\epsilon=\ov{0},\ov{1}$\\
$\osp(m|n)$&  ortho-symplectic Lie superalgebra preserving a\\
& non-degenerate supersymmetric bilinear form on $\C^{m|n}$\\
$\spo(m|n)$&  symplectic-orthogonal Lie superalgebra preserving a\\
 &   non-degenerate skew-supersymmetric bilinear form on $\C^{m|n}$\\
$U(\ga)$&  universal enveloping algebra of a Lie (super)algebra $\ga$\\
$K(\la)$&  Kac module of highest weight $\la$\\
$L(\ga,\la)$ or $L(\la)$ &  irreducible $\ga$-module of highest
weight $\la$ \\
$\text{ch}M$&  character of the $\h$-semisimple module $M$\\
$\rho_0,\rho_1$&  half sums of positive even and odd roots, respectively\\
$\rho$& $\rho_0-\rho_1$\\
$\mf{S}_d$&  symmetric group in $d$ letters\\
$S^\la$&  Specht module corresponding to the partition $\la$\\
$\Delta(\la)$&  (parabolic) Verma $\ga$-module of highest weight $\la$\\
$\mu'$&  conjugate partition of the partition $\mu$\\
$\Pdmn$&  set of $(m|n)$-hook partitions of size $d$ (Definition
    \ref{hook:def})\\
$\Pmn$& $\bigcup_{d\ge 0}\Pdmn$\\
$s_\la(\x)$&  Schur function in $\x$ for the partition $\la$\\
$hs_\la(\x,\y)$&  hook Schur function for the hook
    partition $\la$ (\ref{mn:hook})\\
$\diamondsuit_r$, $\diamondsuit_{k,r}$&  row determinants given in
    (\ref{deltar}) and (\ref{eq_det}), respectively\\
$\mf{a}^+_\infty$&  (one-sided) infinite-rank Kac-Moody Lie algebras of type
    $A$\\
$\mf{b}_\infty, \mf{c}_\infty, \mf{d}_\infty$&  infinite-rank Kac-Moody Lie
    algebras of types $B,C,D$ \\
$\La^{\mf c}_i$&  $i$th fundamental weight of $\mf c_\infty$\\
$\la_w$&  see (\ref{lambdaw})\\
${\mathcal G}^{\mf x},\ov{{\mathcal G}}^{\mf
x},\widetilde{{\mathcal G}}^{\mf x}$& Lie (super)algebras
    corresponding to master diagrams of\\ &Section \ref{masters} of type $\mf{x}$;
    also denoted by ${\mathcal G},\ov{{\mathcal G}},\widetilde{{\mathcal G}}$, respectively\\
$\mc O, \ov{\mc O}, \widetilde{\mc O}$&  parabolic categories of
${\mathcal G}$-,
    $\ov{\mathcal G}$- and $\wt{\mathcal G}$-modules, respectively \\ &(Sections \ref{sec:O} and
    \ref{sec:allO})\\
$\ov{\Delta}(\la),\wt{\Delta}(\la)$&  parabolic Verma modules in
    $\ov{\mc O}, \widetilde{\mc O}$
    (Section \ref{Tfunctors})\\
$\ov{L}(\la),\wt{L}(\la)$&  irreducible modules in $\ov{\mc O},
    \widetilde{\mc O}$ (Section
    \ref{Tfunctors})\\
\end{tabular}

\label{table2}
\begin{tabular}{r l}
{\bf Symbol}& {\bf Meaning}\\
\hline\hline\\
\makebox(23,0){$\oval(20,12)$}\makebox(-20,8){$\mf{x}$}&  head Dynkin diagrams
    of types $\mf{x=a,b,c,d}$ (Section \ref{masters})\\
$P^d_Y, P^d_{\ov Y}, P^d_{\wt Y}$&  set of dominant weights for
    ${\mathcal G},\ov{{\mathcal G}},\widetilde{{\mathcal G}}$, respectively (Section \ref{sec:allO})\\
$\natural,\theta$&  bijections from $P^d_Y$ to $P^d_{\ov Y}$ and $P^d_{\wt Y}$ (Section
    \ref{sec:allO}), respectively\\
$T$, or $\ov{T}$&  functors from $\wt{\mc{O}}$ to ${\mc O}$, or
    to $\ov{\mc O}$, respectively (Section \ref{Tfunctors})\\
$\ov{\mc O}_n$&  parabolic category of a finite-dimensional Lie
    (super)algebra\\
    & (Section \ref{char})\\
$\mf{Tr}_n$&  truncation functor from $\ov{\mc{O}}$ to $\ov{\mc O}_n$ (Section
    \ref{char})\\
\end{tabular}

\bigskip
\frenchspacing

\end{document}